\documentclass[11pt]{amsart}
\usepackage{amsmath,amssymb,latexsym,soul,cite}
\usepackage{color,enumitem,graphicx}
\usepackage[colorlinks=true,urlcolor=blue,
citecolor=red,linkcolor=blue,linktocpage,pdfpagelabels,
bookmarksnumbered,bookmarksopen]{hyperref}
\usepackage[english]{babel}
\usepackage[T1]{fontenc}
\usepackage[hyperpageref]{backref}

\usepackage[left=3.44cm,right=3.44cm,top=3.25cm,bottom=3.25cm]{geometry}

\newcommand{\restr}[2]{\left.#1\right|_{#2}}
\newcommand{\abs}[1]{\left|#1\right|}

\numberwithin{equation}{section}

\newtheorem{thm}{Theorem}[section]
  \theoremstyle{plain}
  \newtheorem{lem}[thm]{Lemma}
  \theoremstyle{plain}
  \newtheorem{prop}[thm]{Proposition}
  \theoremstyle{plain}
  \newtheorem{cor}[thm]{Corollary}
  \theoremstyle{plain}
  \newtheorem{definition}[thm]{Definition}
    \theoremstyle{definition}
\newtheorem{rem}[thm]{Remark}

\newcommand{\N}{{\mathbb N}}
\newcommand{\Z}{{\mathbb Z}}
\newcommand{\R}{{\mathbb R}}
\newcommand{\eps}{\varepsilon}

\newcommand{\M}{{\mathcal M}}

\newcommand{\bdry}[1]{\partial #1}

\newcommand{\isom}{\approx}
\newcommand{\norm}[2][]{\left\|#2\right\|_{#1}}
\renewcommand{\o}{\text{o}}
\newcommand{\PS}[1]{$(\text{PS})_{#1}$}
\newcommand{\QED}{\mbox{\qedhere}}
\newcommand{\seq}[1]{\left(#1\right)}
\newcommand{\set}[1]{\left\{#1\right\}}

\newenvironment{enumarab}{\begin{enumerate}

}{\end{enumerate}}

\newenvironment{enumroman}{\begin{enumerate}

}{\end{enumerate}}

\title[Fractional $p$-Laplacian problems]{Existence results for fractional \\ $p$-Laplacian
problems via Morse theory}

\author[A.\ Iannizzotto]{Antonio Iannizzotto}
\author[S.\ Liu]{Shibo Liu}
\author[K.\ Perera]{Kanishka Perera}
\author[M.\ Squassina]{Marco Squassina}

\address{Dipartimento di Informatica
\newline\indent
Universit\`a degli Studi di Verona
\newline\indent
Strada Le Grazie I-37134 Verona, Italy}
\email{antonio.iannizzotto@univr.it}

\address{School of Mathematical Sciences
\newline\indent
Xiamen University
\newline\indent
Xiamen 361005, China}
\email{liusb@xmu.edu.cn}

\address{Department of Mathematical Sciences
\newline\indent
Florida Institute of Technology
\newline\indent
150 W University Blvd, Melbourne, FL 32901, USA}
\email{kperera@fit.edu}

\address{Dipartimento di Informatica
\newline\indent
Universit\`a degli Studi di Verona
\newline\indent
Strada Le Grazie I-37134 Verona, Italy}
\email{marco.squassina@univr.it}


\begin{document}

\begin{abstract}
We investigate a class of quasi-linear nonlocal problems, including
as a particular case semi-linear problems involving the fractional Laplacian
and arising in the framework of continuum mechanics, phase transition phenomena,
population dynamics and game theory.
Under different growth assumptions on the reaction term, we obtain
various existence as well as finite multiplicity results by means of variational and topological
methods and, in particular, arguments from Morse theory.
\end{abstract}

\maketitle


\bigskip
\begin{center}
\begin{minipage}{10cm}
\footnotesize
\tableofcontents
\end{minipage}
\end{center}

\bigskip
\smallskip


\section{Introduction}
\subsection{General overview}
Let $\Omega$ be a bounded domain in $\R^N$, $N \ge 2$, with Lipschitz boundary $\bdry{\Omega}$. Recently, much attention has been paid 
to the semi-linear problem
\begin{equation}
\begin{cases}
(- \Delta)^s\, u  = f(x,u)  &\text{in $\Omega$}   \\
u = 0 & \text{in $\R^N \setminus \Omega$,}
\end{cases}
\end{equation}
from the point of view of existence, nonexistence and regularity, where
$f$ is a Carath\'{e}odory function satisfying suitable growth conditions.
Several existence results via variational methods are proved in a series of papers of {\sc Servadei $\&$ Valdinoci} \cite{SV, SV1, SV2, SV3} (see also {\sc Iannizzotto $\&$ Squassina} \cite{IS1} for the special case $s=1/2$, $p=2$ and $N=1$, with exponential nonlinearity). The issues of regularity and non-existence of solutions are examined by {\sc Caffarelli $\&$ Silvestre} \cite{CS}, {\sc Ros Oton $\&$ Serra} \cite{RS, RS1,RS2}. The corresponding equation in $\R^N$ is studied by {\sc Cabr\'e \& Sire} \cite{CS1, CS2}. Although the fractional Laplacian operator
$(-\Delta)^s$, and more generally pseudodifferential operators, have been a classical topic in harmonic analysis and partial differential equations for a long time,
the interest in such operators has constantly increased during the last few years.
Nonlocal operators such as $(-\Delta)^s$ naturally arise
in continuum mechanics, phase transition phenomena,
population dynamics and game theory, see e.g.\ {\sc Caffarelli} \cite{C} and the references therein.
In the works of {\sc Metzler $\&$ Klafter} \cite{MK,MK1}, the description of anomalous diffusion
via fractional dynamics is investigated and various fractional partial differential
equations are derived from {\em L\'evy random walk} models, extending
{\em Brownian walk} models in a natural way. In particular, in the paper of {\sc Laskin} \cite{L} a fractional Schr\"odinger
equation was obtained, which extends to a L\'evy framework the classical result that path integral over Brownian trajectories leads to
the Schr\"odinger equation. Fractional operators are also involved in financial mathematics, since L\'ewy processes with jumps
revealed as more appropriate models of stock pricing, compared to the
Brownian ones used in the celebrated Black $\&$ Sholes option pricing model (see {\sc Applebaum} \cite{A}).
\vskip2pt
\noindent
Very recently, a new nonlocal and nonlinear operator was considered, namely for $p \in (1,\infty)$, $s \in (0,1)$ and $u$ smooth enough
\begin{equation}
\label{plap}
(- \Delta)_p^s\, u(x) = 2\, \lim_{\varepsilon \searrow 0} \int_{\R^N \setminus B_\varepsilon(x)} \frac{|u(x) - u(y)|^{p-2}\, (u(x) - u(y))}{|x - y|^{N+sp}}\, dy, \quad x \in \R^N,
\end{equation}
consistent, up to some normalization constant depending upon $N$ and $s$, with the linear fractional Laplacian $(-\Delta)^s$ in the case $p=2$. For the motivations that lead to 
the study of such operators, we refer the reader again to the review paper \cite{C}. This operator, known as the {\em fractional $p$-Laplacian}, leads naturally to the study of the quasi-linear problem
\begin{equation} \label{prob}
\begin{cases}
(- \Delta)_p^s\, u  = f(x,u) & \text{in $\Omega$} \\
u  = 0 & \text{in $\R^N \setminus \Omega$.}
\end{cases}
\end{equation}
One typical feature of the aforementioned operators is the {\em nonlocality}, in the sense that the value of $(-\Delta)^s_p u(x)$ at any point $x\in\Omega$ depends not only on the values of $u$ on the whole $\Omega$, but actually on the whole $\R^N$, since $u(x)$ represents the expected value of a random variable tied to a process randomly jumping arbitrarily far from the point $x$. While in the classical case, by the continuity properties of the Brownian motion, at the exit time from $\Omega$ one necessarily is on $\partial\Omega$, due to the jumping nature of the process,
at the exit time one could end up anywhere outside $\Omega$. In this sense, the natural non-homogeneous Dirichlet boundary condition consists in assigning the values of $u$ in $\R^N\setminus\Omega$ rather than mererly on $\partial\Omega$.
Then, it is reasonable to search for solution in the space of functions $u\in W^{s,p}(\R^N)$ vanishing on the outside of $\Omega$.
It should be pointed out that, in a bounded domain, this is not the only possible
way of providing a formulation of the problem.
\vskip2pt
\noindent
In the works of {\sc Franzina \& Palatucci} \cite{FP} and of {\sc Lindgren \& Linqvist} \cite{LL}, the eigenvalue problem associated with
$(-\Delta)^s_p u$ is studied, and particularly some properties of the first eigenvalue and of the higher order (variational) eigenvalues are obtained. Then, {\sc Iannizzotto \& Squassina} \cite{IS} 
obtained some Weyl-type estimates for the asymptotic behaviour of variational eigenvalues $\lambda_j$ defined by a suitable cohomological index. From the point of view of regularity
theory, some results can be found in \cite{LL} even though that work is most focused on the case where $p$ is large and the solutions
inherit some regularity directly from the functional embeddings themselves. More recently {\sc Di Castro, Kuusi \& Palatucci} \cite{DKP} and {\sc Brasco \& Franzina} \cite{BF} obtained relevant results about the local boundedness and
H\"older continuity for the solutions to the problem of finding $(s,p)$-harmonic functions $u$, that is $(-\Delta)^s_p u=0$ in $\Omega$ with $u=g$ on $\R^N\setminus\Omega$,
for some function $g$, providing an extension of results by De Giorgi-Nash-Moser to the nonlocal nonlinear framework.
Finally, in the work of {\sc Bjorland, Caffarelli \& Figalli} \cite{BC}, some higher regularity is obtained when $s$ gets close to $1$, by showing that the solutions converge
to the solutions with the $p$-Laplace 
operator ${\rm div}(|\nabla u|^{p-2}\nabla u)$, whenever $s\to 1$.

\subsection{Plan of the paper}
In the present paper, we aim at establishing existence and (finite) multiplicity of the weak solutions to \eqref{prob} by making use of advanced tools of Morse theory. The contents of the paper are as follows:
\begin{itemize}
\item In Section~\ref{preliminary}, we introduce some preliminary notions and notations
and set the functional framework of the problem. More precisely, in Subsection \ref{varform} we establish the variational setting for problem \eqref{prob}, in Subsection \ref{eigenvaluep} we recall some basic features about the variational eigenvalues of the operator $(-\Delta)^s_p$ and related topics, and in Subsection \ref{morse} we introduce critical groups and some related notions.
\item In Section~\ref{inftybb} we establish a priori $L^\infty$-bounds for the solutions of problem \eqref{prob} under suitable growth conditions on the nonlinearity. These regularity results
are used also in the existence theorems proved in the subsequent sections. To our knowledge, $L^\infty$-bounds were previously obtained only for the eigenvalue problem
$(-\Delta)^s_pu=\lambda |u|^{p-2}u$, see \cite{FP}. The main result of this section
is Theorem~\ref{3main}.
\item In Section~\ref{psuper}, we deal with the {\em $p$-superlinear case}, namely $f(x,t)=\lambda|t|^{p-2}t+g(x,t)$, with $g(x,\cdot)$ vanishing at zero, proving via Morse-theoretical methods the existence of non-zero solutions for all values
of the real parameter $\lambda$. The main result of this section is Theorem~\ref{4main}.
\item In Section~\ref{coercive}, we deal with the {\em coercive case}, including the case when $f(x,\cdot)$ is $p$-sublinear at infinity, proving via truncations the existence of a positive solution $u_+$ and of a negative solution $u_-$ and the computation of critical groups at zero yields the existence of a third non-zero solution.  The main result of this section is Theorem~\ref{5main}.
\item In Section~\ref{asymptotically}, we deal with the {\em asymptotically $p$-linear case}, namely $f(x,t)=\lambda|t|^{p-2}t+g(x,t)$ with $g(x,\cdot)$ vanishing at infinity, proving some existence results via the computation of critical groups at infinity and a multiplicity result, for $\lambda$ large enough, via the Mountain Pass Theorem.
The main results of this section are Theorems~\ref{6ex1}, \ref{6ex2} and \ref{6mult}.
\item In Section~\ref{pohoz}, we discuss Poho\v{z}aev identity and consequent nonexistence results in star-shaped domains.
\end{itemize}
For a short introduction to fractional Sobolev spaces, we shall refer to the Hitchhiker's guide of {\sc Di Nezza, Palatucci \& Valdinoci} \cite{DPV}. Concerning the Morse-theoretic apparatus, topological tools as well as existence and multiplicity results
for the local case $s=1$, we shall refer the reader to the monograph of {\sc Perera, Agarwal \& O'Regan} \cite{PAO}, to the classical books by {\sc Chang} \cite{C1}, {\sc Mawhin \& Willem} \cite{MW}, {\sc Milnor} \cite{M} and to the references therein.
\vskip6pt
\noindent
{\bf Acknowledgements.} Shibo Liu was supported by National Natural Science Foundation of China (No.\ 11171204). Marco Squassina was partially  supported by 2009 MIUR project:
Variational and Topological Methods in the Study of Nonlinear Phenomena.
The authors would like to thank Xavier Ros-Oton for precious bibliographic information on the regularity up to the boundary
of the solutions to the problem, as well as Sun-Ra Mosconi for some useful remarks concerning Section~\ref{coercive}.

\section{Preliminaries}
\label{preliminary}

\noindent
In this preliminary section, for the reader's convenience, we collect some basic results that will be used in the forthcoming sections. In the following, for any functional $\Phi$ and any Banach space $(X,\|\cdot\|)$ we will denote
\[\Phi^c=\{u\in X\,:\,\Phi(u)\le c\} \ (c\in\R),\]
\[\overline B_\rho(u_0)=\{u\in X\,:\,\|u-u_0\|\le\rho\} \ (u_0\in X,\, \rho>0).\]
Moreover, in the proofs of our results, $C$ will denote a positive constant (whose value may change case by case).

\subsection{Variational formulation of the problem}
\label{varform}

\noindent
Let $\Omega\subset\R^N$ be a bounded domain with smooth boundary $\partial\Omega$, and for all $1\le\nu\le\infty$ denote by $\|\cdot\|_\nu$ the norm of $L^\nu(\Omega)$. Moreover, let $0<s<1<p<\infty$ be real numbers, and the fractional critical exponent be defined as
\[p^*_s= \begin{cases}
\frac{Np}{N-sp} & \mbox{if $sp<N$} \\
\infty & \mbox{if $sp\ge N$}.
\end{cases}\]
First we introduce a variational setting for problem \eqref{prob}. The Gagliardo seminorm is defined for all measurable function $u:\R^N\to\R$ by
\[[u]_{s,p} = \Big(\int_{\R^{2N}} \frac{|u(x) - u(y)|^p}{|x - y|^{N+sp}}{\rm d}x{\rm d}y\Big)^{1/p}.\]
We define the fractional Sobolev space
\[W^{s,p}(\R^N)=\{u\in L^p(\R^N)\,:\,u \ \mbox{measurable, $[u]_{s,p}<\infty$}\},\]
endowed with the norm
\[\|u\|_{s,p}=\big(\|u\|_p^p+[u]_{s,p}^p\big)^\frac{1}{p}.\]
For a detailed account on the properties of $W^{s,p}(\R^N)$ we refer the reader to \cite{DPV}. We shall work in the closed linear subspace
\[X(\Omega)=\{u\in W^{s,p}(\R^N)\,:\,u(x)=0 \ \mbox{a.e. in $\R^N\setminus\Omega$}\},\]
which can be equivalently renormed by setting $\|\cdot\|=[\,\cdot\,]_{s,p}$ (see \cite[Theorem 7.1]{DPV}). It is readily seen that $(X(\Omega),\|\cdot\|)$ is a uniformly convex Banach space and that the embedding $X(\Omega)\hookrightarrow L^\nu(\Omega)$ is continuous for all $1\le\nu\le p^*_s$, and compact for all $1\le\nu<p^*_s$ (see \cite[Theorems 6.5, 7.1]{DPV}). The dual space of $(X(\Omega),\|\cdot\|)$ is denoted by $(X(\Omega)^*,\|\cdot\|_*)$.
\vskip2pt
\noindent
We rephrase variationally the fractional $p$-Laplacian as the nonlinear operator $A:X(\Omega)\to X(\Omega)^*$ defined for all $u,v\in X(\Omega)$ by
\[\langle A(u),v\rangle=\int_{\R^{2N}} \frac{|u(x) - u(y)|^{p-2}(u(x)-u(y))(v(x)-(y))}{|x - y|^{N+sp}}{\rm d}x{\rm d}y.\]
It can be seen that, if $u$ is smooth enough, this definition coincides with that of \eqref{plap}. A {\em (weak) solution} of problem \eqref{prob} is a function $u\in X(\Omega)$ such that
\begin{equation}\label{weak}
\langle A(u),v\rangle=\int_\Omega f(x,u)v \, {\rm d}x
\end{equation}
for all $v\in X(\Omega)$.
\vskip2pt
\noindent
Clearly, $A$ is odd, $(p-1)$-homogeneous, and satisfies for all $u\in X(\Omega)$
\[\langle A(u),u\rangle=\|u\|^p, \quad \|A(u)\|_*\le\|u\|^{p-1}.\]
Since $X(\Omega)$ is uniformly convex, by \cite[Proposition 1.3]{PAO}, $A$ satisfies the following compactness condition:
\begin{itemize}
\item[$({\bf S})$] If $(u_n)$ is a sequence in $X(\Omega)$ such that $u_n\rightharpoonup u$ in $X(\Omega)$ and $\langle A(u_n),u_n-u\rangle\to 0$, then $u_n\to u$ in $X(\Omega)$.
\end{itemize}
Moreover, $A$ is a potential operator, precisely $A$ is the G\^ateaux derivative of the functional $u\mapsto\|u\|^p/p$ in $X(\Omega)$. Thus, $A$ satisfies all the structural assumptions of \cite{PAO}.
\vskip2pt
\noindent
Now we introduce the minimal hypotheses on the reaction term of \eqref{prob}:
\begin{itemize}
\item[${\bf H}_2$] $f:\Omega\times\R\to\R$ is a Carath\'eodory mapping, $F(x,t)=\int_0^t f(x,\tau){\rm d}\tau$ for all $(x,t)\in\Omega\times\R$, and
\[|f(x,t)|\le a(1+|t|^{r-1})\]
a.e. in $\Omega$ and for all $t\in\R$ ($a>0$, $1<r<p^*_s$).
\end{itemize}
We set for all $u\in X(\Omega)$
\begin{equation}\label{2phi}
\Phi(u)=\frac{\|u\|^p}{p}-\int_\Omega F(x,u){\rm d}x.
\end{equation}
By ${\bf H}_2$, we have $\Phi\in C^1(X(\Omega))$. We denote by $K(\Phi)$ the set of all critical points of $\Phi$. If $u\in K(\Phi)$, then \eqref{weak} holds for all $v\in X(\Omega)$, i.e., $u$ is a weak solution of \eqref{prob}. We recall now the Palais-Smale and the Cerami compactness conditions in a set $U\subseteq X$:
\begin{itemize}
\item[${\bf PS}$] every sequence $(u_n) $ in $U$ such that $(\Phi(u_n))$ is bounded in $\R$ and $\Phi'(u_n)\to 0$ in $X(\Omega)^*$ admits a convergent subsequence;
\item[${\bf C}$] every sequence $(u_n) $ in $U$ such that $(\Phi(u_n))$ is bounded in $\R$ and $(1+\|u_n\|)\Phi'(u_n)\to 0$ in $X(\Omega)^*$ admits a convergent subsequence.
\end{itemize}
Such conditions hold for our $\Phi$, provided that the boundedness of the sequence is assumed:

\begin{prop}\label{2ps}
If ${\bf H}_2$ holds, and every sequence $(u_n) $ in $X(\Omega)$ such that $\Phi'(u_n)\to 0$ (respectively, $(1+\|u_n\|)\Phi'(u_n)\to 0$) in $X(\Omega)^*$ is bounded, then $\Phi$ satisfies ${\bf PS}$ (respectively, ${\bf C}$) in $X(\Omega)$.
\end{prop}
\begin{proof}
We deal with ${\bf PS}$. Passing to a relabeled subsequence, we have $u_n\rightharpoonup u$ in $X(\Omega)$, and $u_n\to u$ in $L^r(\Omega)$. So we have for all $n\in\N$
\begin{align*}
|\langle A(u_n),u_n-u\rangle| &= \Big|\langle\Phi'(u_n),u_n-u\rangle+\int_\Omega f(x,u_n)(u_n-u){\rm d}x\Big| \\
&\le \|\Phi'(u_n)\|_*\|u_n-u\|+\int_\Omega(1+|u_n|^{r-1})|u_n-u|{\rm d}x \\
&\le \|\Phi'(u_n)\|_*\|u_n-u\|+C(1+\|u_n\|_r^{r-1})\|u_n-u\|_r,
\end{align*}
and the latter tends to $0$ as $n\to\infty$. So, by the $({\bf S})$-property of $A$, we have $u_n\to u$ in $X(\Omega)$.
\end{proof}

\noindent
The following strong maximum principle (see \cite[Theorem A.1]{BF}, a consequence of \cite[Lemma 1.3]{DKP}) will be useful in the proof of some of our results:

\begin{prop}\label{mp}
If $u\in X(\Omega)\setminus\{0\}$ is such that $u(x)\geq 0$ a.e. in $\Omega$ and
\[\langle A(u),v\rangle\geq 0\]
for all $v\in X(\Omega)$, $v(x)\geq 0$ a.e. in $\Omega$, then $u(x)>0$ a.e. in $\Omega$.
\end{prop}

\subsection{An eigenvalue problem}
\label{eigenvaluep}

We consider the nonlinear eigenvalue problem
\begin{equation}\label{eigen}
\left\{  \begin{array}{ll}
    (- \Delta)_p^s\, u=\lambda|u|^{p-2}u &\mbox{in $\Omega$} \\
    u=0 &\mbox{on $\R^N\setminus\Omega$},
        \end{array}\right.
\end{equation}
depending on the parameter $\lambda\in\R$. If \eqref{eigen} admits a weak solution $u\in X(\Omega)\setminus\{0\}$, then $\lambda$ is an {\em eigenvalue} and $u$ is a $\lambda$-{\em eigenfunction}. The set of all eigenvalues is referred to as the {\em spectrum} of $(-\Delta)^s_p$ in $X(\Omega)$ and denoted by $\sigma(s,p)$. As in the classical case of the $p$-Laplacian, the structure of $\sigma(s,p)$ is not completely known yet, but many properties have been detected by several authors, see for instance \cite{FP, IS, LL}. Here we recall only the results that we will use in the forthcoming sections.
\vskip2pt
\noindent
We already know from continuous embedding that the Rayleigh quotient
\begin{equation}\label{rayquot}
\lambda_1=\inf_{u\in X(\Omega)\setminus\{0\}}\frac{\|u\|^p}{\|u\|_p^p}
\end{equation}
lies in $(0,\infty)$. The number $\lambda_1$ plays an important role in the study of problem \eqref{eigen}. We list below some spectral properties of $(-\Delta)^s_p$:

\begin{prop}\label{spec}
The eigenvalues and eigenfunctions of \eqref{eigen} have the following properties:
\begin{itemize}
\item[$(i)$] $\lambda_1={\rm min}\, \sigma(s,p)$ is an isolated point of $\sigma(s,p)$;
\item[$(ii)$] all $\lambda_1$-eigenfunctions are proportional, and if $u$ is a $\lambda_1$-eigenfunction, then either $u(x)>0$ a.e. in $\Omega$ or $u(x)<0$ a.e. in $\Omega$;
\item[$(iii)$] if $\lambda\in\sigma(s,p)\setminus\{\lambda_1\}$ and $u$ is a $\lambda$-eigenfunction, then $u$ changes sign in $\Omega$;
\item[$(iv)$] all eigenfunctions are in $L^\infty(\Omega)$;
\item[$(v)$] $\sigma(s,p)$ is a closed set.
\end{itemize}
\end{prop}

\noindent
We define a non-decreasing sequence $(\lambda_k)$ of {\em variational eigenvalues} of $(-\Delta)^s_p$ by means of the cohomological index. This type of construction was introduced for the $p$-Laplacian by {\sc Perera} \cite{P} (see also {\sc Perera $\&$ Szulkin} \cite{PS}), and it is slightly different from the traditional one, based on the Krasnoselskii genus (which does not give the additional Morse-theoretical information that we need here).
\vskip2pt
\noindent
We briefly recall the definition of $\Z_2$-cohomological index by {\sc Fadell $\&$ Rabinowitz} \cite{FR}. For any closed, symmetric subset $M$ of a Banach space $X$, let $\overline M=M/\Z_2$ be the quotient space (in which $u$ and $-u$ are identified), and let $\phi:\overline M\to\R {\rm P}^\infty$ be the classifying map of $\overline M$, which induces a homomorphism $\phi^*:H^*(\R {\rm P}^\infty)\to H^*(\overline M)$ of the Alexander-Spanier cohomology rings with coefficients in $\Z_2$. We may identify $H^*(\R {\rm P}^\infty)$ with the polynomial ring $\Z_2[\omega]$. The {\em cohomological index} of $M$ is then
\[i(M) = \begin{cases}
\sup\{k\in\N\,:\,\phi^*(\omega^k)\ne 0\} & \mbox{if $M\ne\emptyset$}\\
0 & \mbox{if $M=\emptyset$}.
\end{cases}\]
Now let us come back to our case. We set for all $u\in X(\Omega)$
\[J(u)=\frac{\|u\|_p^p}{p}, \quad I(u)=\frac{\|u\|^p}{p}, \ \Psi(u)=\frac{1}{J(u)} \ (u\ne 0)\]
and define a $C^1$-Finsler manifold by setting
\begin{equation}\label{2manifold}
\mathcal{M}=\{u\in X(\Omega)\,:\, I(u)=1\}.
\end{equation}
For all $k\in\N$, we denote by $\mathcal{F}_k$ the family of all closed, symmetric subsets $M$ of $\mathcal{M}$ such that $i(M)\ge k$, and set
\begin{equation}\label{minimax}
\lambda_k=\inf_{M\in\mathcal{F}_k}\sup_{u\in M}\Psi(u)
\end{equation}
(note that, for $k=1$, \eqref{rayquot} and \eqref{minimax} agree). For all $k\in\N$, $\lambda_k$ turns out to be a critical value of the restricted functional $\restr{\Psi}{\mathcal{M}}$ (which is even and satisfies ${\bf PS}$ by \cite[Lemma 4.5]{PAO}), hence, by the Lagrange multiplier rule, an eigenvalue of $(-\Delta)^s_p$. These eigenvalues have the following remarkable properties (see \cite[Theorem 4.6]{PAO}):

\begin{prop}\label{2lambdak}
The sequence $(\lambda_k)$ defined by \eqref{minimax} is non-decreasing and $\lambda_k\to\infty$ as $k\to\infty$. Moreover, for all $k\in\N$ we have
\[i\big(\{u\in\mathcal{M}\,:\,\Psi(u)\le\lambda_k\}\big)=i\big(\{u\in\mathcal{M}\,:\,\Psi(u)<\lambda_{k+1}\}\big)=k\]
\end{prop}

\begin{rem}\label{alteigen}
In \cite{IS} a different construction of the variational eigenvalues is performed. Such construction is equivalent to that described above, up to a point: precisely, one can easily see that, following the method of \cite{IS}, we obtain exactly the same sequence $(\lambda_k)$, while it is not certain whether the topological property in Proposition \ref{2lambdak} holds, or not.
\end{rem}

\subsection{Critical groups}
\label{morse}

We recall the definition and some basic properties of critical groups, referring the reader to the monograph \cite{PAO} for a detailed account on the subject. Let $X$ be a Banach space, $\Phi\in C^1(X)$ be a functional satisfying ${\bf C}$, and denote by $K(\Phi)$ the set of all critical points of $\Phi$. Let $u\in X$ be an {\em isolated} critical points of $\Phi$, i.e., there exists a neighborhood $U$ of $u$ such that $K(\Phi)\cap U=\{u\}$, and $\Phi(u)=c$. For all $k\in\N_0$, the $k$-{\em th (cohomological) critical group} of $\Phi$ at $u$ is defined as
\[C^k(\Phi,u)=H^k(\Phi^c\cap U,\Phi^c\cap U\setminus\{u\}),\]
where $H^*(M,N)$ denotes again the Alexander-Spaniel cohomology with coefficients in $\Z_2$ for a topological pair $(M,N)$.
\vskip2pt
\noindent
The definition above is well posed, since cohomology groups are invariant under excision, so $C^k(\Phi,u)$ does not depend on $U$. Moreover, critical groups are invariant under homotopies preserving isolatedness of critical points (see {\sc Chang $\&$ Ghoussoub} \cite{CG}, {\sc Corvellec $\&$ Hantoute} \cite{CH}).

\begin{prop}\label{2hominv}
Let $X$ be a Banach space, $u\in X$, and for all $\tau\in[0,1]$ let $\Phi_\tau\in C^1(X)$ be a functional such that $u\in K(\Phi_\tau)$. If there exists a closed neighborhood $U\subset X$ of $u$ such that
\begin{itemize}
\item[$(i)$] $\Phi_\tau$ satisfies ${\bf PS}$ in $U$ for all $\tau\in[0,1]$;
\item[$(ii)$] $K(\Phi_\tau)\cap U=\{u\}$ for all $\tau\in[0,1]$;
\item[$(iii)$] the mapping $\tau\mapsto\Phi_\tau$ is continuous between $[0,1]$ and $C^1(U)$,
\end{itemize}
then for all $k\in\N_0$ we have $C^k(\Phi_1,u)=C^k(\Phi_0,u)$.
\end{prop}

\noindent
We recall some special cases in which the computation of critical groups is immediate ($\delta_{k,h}$ is the Kronecker symbol):

\begin{prop}\label{2critgr}
Let $X$ be a Banach space with ${\rm dim}(X)=\infty$, $\Phi\in C^1(X)$ be a functional satisfying ${\bf C}$, $u\in K(\Phi)$ be an isolated critical point of $\Phi$. The following hold:
\begin{itemize}
\item[$(i)$] if $u$ is a local minimizer of $\Phi$, then $C^k(\Phi,u)=\delta_{k,0}\,\Z_2$ for all $k\in\N_0$;
\item[$(ii)$] if $u$ is a local maximizer of $\Phi$, then $C^k(\Phi,u)=0$ for all $k\in\N_0$.
\end{itemize}
\end{prop}

\noindent
If the set of critical values of $\Phi$ is bounded below, we define for all $k\in\N_0$ the $k$-{\em th critical group at infinity} of $\Phi$ as
\[C^k(\Phi,\infty)=H^k(X,\Phi^\eta),\]
where $\eta<\inf_{u\in K(\Phi)}\Phi(u)$. We recall the {\em Morse identity}:

\begin{prop}\label{morseid}
Let $X$ be a Banach space, $\Phi\in C^1(X)$ be a functional satisfying ${\bf C}$, such that $K(\Phi)$ is a finite set. Then, there exists a formal power series $Q(t)=\sum_{k=0}^\infty q_k t^k$ ($q_k\in\N_0$ for all $k\in\N_0$) such that for all $t\in\R$
\[\sum_{k=0}^\infty\sum_{u\in K(\Phi)}{\rm rank}\,C^k(\Phi,u)t^k=\sum_{k=0}^\infty{\rm rank}\,C^k(\Phi,\infty)t^k+(1+t)Q(t).\]
\end{prop}

\noindent
In the absence of a direct sum decomposition, one of the main technical tools that we use to compute the critical groups of $\Phi$ at zero is the notion of a {\em cohomological local splitting} introduced in \cite{PAO}, which is a variant of the homological local linking of {\sc Perera} \cite{P1}. The following slightly different form of this notion was given in {\sc Degiovanni, Lancelotti $\&$ Perera} \cite{DLP}.

\begin{definition}\label{2clsplit}
A functional $\Phi\in C^1(X)$ has a cohomological local splitting near $0$ in dimension $k\in\N$, if there exist symmetric cones $X_\pm \subset X$ with $X_+ \cap X_- =\{0\}$ and $\rho > 0$ such that
\begin{itemize}
\item[$(i)$] $i(X_- \setminus \set{0}) = i(X \setminus X_+) = k$;
\item[$(ii)$] $\Phi(u) \le \Phi(0)$ for all $u\in\overline B_\rho(0)\cap X_-$, and $\Phi(u)\ge\Phi(0)$ for all $u\in\overline B_\rho(0) \cap X_+$.
\end{itemize}
\end{definition}

\noindent
In this case, we have the following result (see \cite[Proposition 2.1]{DLP}):

\begin{prop}\label{2degio}
If $X$ is a Banach space and $\Phi \in C^1(X)$ has a cohomological local splitting near $0$ in dimension $k\in\N$, and $0$ is an isolated critical point of $\Phi$, then $C^k(\Phi,0) \ne 0$.
\end{prop}

\section{$L^\infty$-Bounds on the weak solutions}
\label{inftybb}

\noindent
In this section we will prove some {\em a priori} $L^\infty$-bounds on the weak solutions of problem \eqref{prob}. Similar bounds were obtained  before in some special cases, namely for a semilinear fractional Laplacian equation with reaction term independent of $u$ (see \cite[Proposition 7]{SV}), and for the eigenvalue problem \eqref{eigen} (see \cite[Theorem 3.2]{FP}). Our hypothesis on the reaction term is the following:

\begin{itemize}
\item[${\bf H}_3$] $f:\Omega\times\R\to\R$ is a Carath\'eodory mapping satisfying a.e. in $\Omega$ and for all $t\in\R$
\[|f(x,t)|\le a(|t|^{q-1}+|t|^{r-1}),\]
for some $a>0$, $1\le q\le r<p^*_s$.
\end{itemize}

\noindent
The main result of the section is the following:

\begin{thm}\label{3main}
If ${\bf H}_3$ holds with $q\le p\le r$ satisfying
\[1+\frac{q}{p}>\frac{r}{p}+\frac{r}{p^*_s},\]
then there exist $K>0$ and $\alpha>1$, only depending on $s$, $p$, $\Omega$, $a$, $q$, and $r$, such that, for every weak solution $u\in X(\Omega)$ of \eqref{prob}, we have $u\in L^\infty(\Omega)$ and
\[\|u\|_\infty\le K(1+\|u\|_r^\alpha).\]
\end{thm}
\begin{proof}
Fix a weak solution $u\in X(\Omega)$ of \eqref{prob} with $u^+\ne 0$. We choose $\rho\ge\max\{1,\|u\|_r^{-1}\}$, set $v=(\rho\|u\|_r)^{-1}u$, so $v\in X(\Omega)$, $\|v\|_r=\rho^{-1}$, and $v$ is a weak solution of the auxiliary problem
\begin{equation}\label{bound1}
\left\{  \begin{array}{ll}
    (- \Delta)_p^s\, v=(\rho\|u\|_r)^{1-p}f(x,\rho\|u\|_rv) &\mbox{in $\Omega$} \\
    v=0 &\mbox{on $\R^N\setminus\Omega$}.
        \end{array}
      \right.
\end{equation}
For all $n\in\N$ we set $v_n=(v-1+2^{-n})^+$, so $v_n\in X(\Omega)$, $v_0=v^+$, and for all $n\in\N$ we have $0\le v_{n+1}(x)\le v_n(x)$ and $v_n(x)\to (v(x)-1)^+$ a.e. in $\Omega$ as $n\to\infty$. Moreover, the following inclusion holds (up to a Lebesgue null set):
\begin{equation}\label{bound2}
\{v_{n+1}>0\}\subseteq\{0<v<(2^{n+1}-1)v_n\}\cap\{v_n>2^{-n-1}\}.
\end{equation}
For all $n\in\N$ we set $R_n=\|v_n\|_r^r$, so $R_0=\|v^+\|_r^r\le\rho^{-r}$, and $(R_n)$ is a nonincreasing sequence in $[0,1]$. We shall prove that $R_n\to 0$ as $n\to\infty$. By H\"older inequality, the fractional Sobolev inequality (see \cite[Theorem 6.5]{DPV}), \eqref{bound2}, and Chebyshev inequality we have for all $n\in\N$
\begin{align*}
R_{n+1} &\le |\{v_{n+1}>0\}|^{1-\frac{r}{p^*_s}}\|v_{n+1}\|_{p^*_s}^r \\
&\le C|\{v_n^r>2^{-r(n+1)}\}|^{1-\frac{r}{p^*_s}}\|v_{n+1}\|^r \\
&\le C\, 2^{\big(r-\frac{r^2}{p^*_s}\big)(n+1)} R_n^{1-\frac{r}{p^*_s}}\|v_{n+1}\|^r.
\end{align*}
So, what we need now is an estimate of $\|v_{n+1}\|$. Using the elementary inequality
\[|\xi^+-\eta^+|^p\le|\xi-\eta|^{p-2}(\xi-\eta)(\xi^+-\eta^+) \quad (\xi,\eta\in\R),\]
testing \eqref{bound1} with $v_{n+1}$, and applying also \eqref{bound2}, we obtain
\begin{align*}
\|v_{n+1}\|^p &\le \langle A(v),v_{n+1}\rangle \\
&= \int_\Omega (\rho\|u\|_r)^{1-p}f(x,\rho\|u\|_rv)v_{n+1} {\rm d}x \\
&\le C\int_{\{v_{n+1}>0\}}\Big((\rho\|u\|_r)^{q-p}|v|^{q-1}+(\rho\|u\|_r)^{r-p}|v|^{r-1}\Big)v_{n+1} {\rm d}x \\
&\le C(\rho\|u\|_r)^{r-p}\int_{\{v_{n+1}>0\}}\Big((2^{n+1}-1)^{q-1}v_n^q+(2^{n+1}-1)^{r-1}v_n^r\Big) {\rm d}x \\
&\le C\,2^{(r-1)(n+1)}(\rho\|u\|_r)^{r-p}R_n^\frac{q}{r}.
\end{align*}
Concatenating the inequalities above we have
\[R_{n+1}\le C\,2^{\big(r+\frac{r^2}{p}-\frac{r}{p}-\frac{r^2}{p^*_s}\big)(n+1)}(\rho\|u\|_r)^{\frac{r^2}{p}-r}R_n^{1+\frac{q}{p}-\frac{r}{p^*_s}},\]
which rephrases as the recursive inequality
\begin{equation}\label{bound3}
R_{n+1}\le H^n(\rho\|u\|_r)^{\frac{r^2}{p}-r}R_n^{1+\beta},
\end{equation}
where $H>1$ and $0<\beta<1$ only depend on the data of \eqref{prob}. Now we set $\gamma=r\beta+r-r^2/p>0$, and fix
\[\rho=\max\left\{1,\|u\|_r^{-1},\eta^{-\frac{1}{\gamma}}\|u\|_r^{\big(\frac{r^2}{p}-r\big)\frac{1}{\gamma}}\right\}.\]
We prove that, provided $\rho$ is big enough, for all $n\in\N$
\begin{equation}\label{bound4}
R_n\le \frac{\eta^n}{\rho^r},
\end{equation}
for $\eta=H^{-1/\beta}\in (0,1)$. We argue by induction. We already know that $R_0\le\rho^{-r}$. Assuming that \eqref{bound4} holds for some $n\in\N$, by \eqref{bound3} we have
\[R_{n+1} \le H^n(\rho\|u\|_r)^{\frac{r^2}{p}-r}\Big(\frac{\eta^n}{\rho^r}\Big)^{1+\beta} \le \frac{\|u\|_r^{\frac{r^2}{p}-r}\eta^n}{\rho^{\gamma+r}} \le \frac{\eta^{n+1}}{\rho^r}.\]
By \eqref{bound4} we have $R_n\to 0$. This, in turn, implies that $v_n(x)\to 0$ a.e. in $\Omega$, so $v(x)\le 1$ a.e. in $\Omega$. An analogous argument applies to $-v$, so we have $v\in L^\infty(\Omega)$ and $\|v\|_\infty\le 1$, hence $u\in L^\infty(\Omega)$ and
\begin{align*}
\|u\|_\infty &\le \rho\|u\|_r \\
&= \max\left\{\|u\|_r,1,\eta^{-\frac{1}{\gamma}}\|u\|_r^{1+\big(\frac{r^2}{p}-r\big)\frac{1}{\gamma}}\right\} \\
&\le K(1+\|u\|_r^\alpha),
\end{align*}
for some $K>0$ and $\alpha>1$ only depending on the data of \eqref{prob}. This concludes the proof.
\end{proof}

\noindent
If, in ${\bf H}_3$, we assume $q=p$, then we can improve Theorem \ref{3main} in a twofold way: we may take any $r$ below the critical exponent, and the inequality relating the $L^\infty$-norms of solutions to the $L^r$-norms is of linear type.

\begin{cor}\label{3cor}
If ${\bf H}_3$ holds with $q=p\le r<p^*_s$, then for all $0<\eps<1$ there exists $K>0$, only depending on $s$, $p$, $\Omega$, $a$, and $r$, such that, for every weak solution $u\in X(\Omega)$ of \eqref{prob} with $\|u\|_r<K$, we have $u\in L^\infty(\Omega)$ and
\[\|u\|_\infty\le K^{-1}\|u\|_r.\]
\end{cor}
\begin{proof}
Fix $0<\eps<1$. Let $u\in X(\Omega)$ be a weak solution of \eqref{prob} with $u^+\neq 0$ and $\|u\|_r\le\eps$. We set $v=\eps^{-1}u$. Then, $v\in X(\Omega)$ and $\|v\|_r\le 1$. For all $n\in\N$ we set $v_n=(v-1+2^{-k})^+$ and $R_n=\|v_n\|_r^r$. Reasoning as in the proof of Theorem \ref{3main}, we derive the following recursive inequality:
\begin{equation}\label{boundcor1}
R_{n+1}\le H^N R_n^{1+\beta},
\end{equation}
for some $H>1$, $0<\beta<1$ depending only on the data of \eqref{prob}. We set $\eta=H^{-1/\beta}\in (0,1)$ and $\delta=\eta^{1/(\beta r)}\eps\in(0,\eps)$. If $\|u\|_r=\delta$, then  for all $n\in\N$ we have
\begin{equation}\label{boundcor2}
R_n\le\frac{\delta^r}{\eps^r}\eta^n.
\end{equation}
Indeed, clearly $R_0\le\delta^r/\eps^r$. Moreover, if \eqref{boundcor2} holds for some $n\in\N$, then by \eqref{boundcor1} we have
\[R_{n+1}\le H^n\Big(\frac{\delta^r}{\eps^r}\eta^n\Big)^{1+\beta}=\frac{\delta^r}{\eps^r}\eta^{n+1}.\]
By \eqref{boundcor2} we have $R_n\to 0$ as $n\to\infty$, so $v(x)\le 1$ a.e. in $\Omega$. Reasoning in a similar way on $-v$, we get $\|v\|_\infty\le 1$, hence
\[\|u\|_\infty\le\eps=\eta^{-\frac{1}{\beta r}}\|u\|_r.\]
We set $K=\eta^{1/\beta r}$. Letting $\eps$ span the interval $(0,1)$, we see that for every weak solution $u\in X(\Omega)$ of \eqref{prob} with $\|u\|_r<K$ we have $u\in L^\infty(\Omega)$ and $\|u\|_\infty\le K^{-1}\|u\|_r$.
\end{proof}

\section{$p$-Superlinear case}
\label{psuper}

\noindent
In this section we study problem \eqref{prob}, rephrased as
\begin{equation}\label{lprob}
\left\{  \begin{array}{ll}
    (- \Delta)_p^s\, u=\lambda|u|^{p-2}u+g(x,u) &\mbox{in $\Omega$} \\
    u=0 &\mbox{on $\R^N\setminus\Omega$},
        \end{array}\right.
\end{equation}
where $\lambda\in\R$ is a parameter and the hypotheses on the reaction term are the following:
\begin{itemize}
\item[${\bf H}_4$] $g:\Omega\times\R\to\R$ is a Carath\'eodory mapping, $G(x,t)=\int_0^t g(x,\tau){\rm d}\tau$, and
\begin{itemize}
\item[$(i)$] $|g(x,t)|\le a(1+|t|^{r-1})$ a.e. in $\Omega$ and for all $t\in\R$ ($a>0$, $p<r<p^*_s$);
\item[$(ii)$] $0<\mu G(x,t)\le g(x,t)t$ a.e. in $\Omega$ and for all $|t|\ge R$ ($\mu>p$, $R>0$);
\item[$(iii)$] $\displaystyle\lim_{t\to 0}\frac{g(x,t)}{|t|^{p-1}}=0$ uniformly a.e. in $\Omega$.
\end{itemize}
\end{itemize}
Since $g(x,\cdot)$ does not necessarily vanish at infinity, hypotheses ${\bf H}_4$ classify problem \eqref{lprob} as $p$-{\em superlinear}. Besides, by ${\bf H}_4(iii)$ we have $g(x,0)=0$ a.e. in $\Omega$, so \eqref{lprob} admits the zero solution for all $\lambda\in\R$. By means of Morse theory and the spectral properties of $(-\Delta)^s_p$, we will prove the existence of a non-zero solution for all $\lambda\in\R$, requiring when necessary additional sign conditions on $G(x,\cdot)$ near zero. Results of this type were first proved for the $p$-Laplacian in \cite{DLP} (see also {\sc Perera \& Sim} \cite{PS1}).
\vskip2pt
\noindent
The main result of this section is the following:

\begin{thm}\label{4main}
If ${\bf H}_4$ and one of the following hold:
\begin{itemize}
\item[$(i)$] $\lambda\notin(\lambda_k)$;
\item[$(ii)$] $\lambda\in(\lambda_k)$ and $G(x,t)\ge 0$ a.e. in $\Omega$ and for all $|t|\le\delta$ (for some $\delta>0$);
\item[$(iii)$] $\lambda\in(\lambda_k)$ and $G(x,t)\le 0$ a.e. in $\Omega$ and for all $|t|\le\delta$ (for some $\delta>0$),
\end{itemize}
then problem \eqref{lprob} admits a non-zero solution.
\end{thm}

\noindent
In the present case, the energy functional takes for all $u\in X(\Omega)$ the form
\[\Phi(u)=\frac{\|u\|^p}{p}-\frac{\lambda\|u\|_p^p}{p}-\int_\Omega G(x,u){\rm d}x.\]

\begin{lem}\label{4ps}
The functional $\Phi\in C^1(X(\Omega))$ satisfies ${\bf PS}$. Moreover, there exists $\eta<0$ such that $\Phi^\eta$ is contractible.
\end{lem}
\begin{proof}
By ${\bf H}_4 (ii)$ we have a.e. in $\Omega$ and for all $t\in\R$
\begin{equation}\label{4a}
G(x,t)\ge C_0|t|^\mu-C_1 \ (C_0,C_1>0).
\end{equation}
Let $(u_n)$ be a sequence in $X(\Omega)$ such that $(\Phi(u_n))$ is bounded in $\R$ and $\Phi'(u_n)\to 0$ in $X(\Omega)^*$. By \eqref{4a} we have for all $n\in\N$
\begin{align*}
\Big(\frac{\mu}{p}-1\Big)\frac{\|u_n\|^p}{2} &= \frac{\mu+p}{2}\Phi(u_n)-\langle\Phi'(u_n),u_n\rangle+\frac{\lambda}{2}\Big(\frac{\mu}{p}-1\Big)\|u_n\|_p^p \\
& +\int_\Omega\Big(\frac{\mu+p}{2}G(x,u_n)-g(x,u_n)u_n\big){\rm d}x \\
&\le \|\Phi'(u_n)\|_*\|u_n\|+\frac{\lambda}{2}\Big(\frac{\mu}{p}-1\Big)\|u_n\|_p^p-\frac{\mu-p}{2}\|u_n\|_\mu^\mu+C \\
&\le \|\Phi'(u_n)\|_*\|u_n\|+C(1+\|u_n\|_\mu^p-\|u_n\|_\mu^\mu),
\end{align*}
hence $(u_n)$ is bounded in $X(\Omega)$. By Proposition \ref{2ps}, $\Phi$ satisfies ${\bf PS}$.
\vskip2pt
\noindent
Now, fix $u\in X(\Omega)\setminus\{0\}$. By \eqref{4a} we have for all $\tau>0$
\[\Phi(\tau u)\le\frac{\tau^p\|u\|^p}{p}-\frac{\lambda\tau^p\|u\|_p^p}{p}-C(\tau^\mu\|u\|_\mu^\mu-1),\]
and the latter tends to $-\infty$ as $\tau\to\infty$. In particular, $\Phi$ is unbounded below in $X(\Omega)$. Moreover, by ${\bf H}_4(ii)$ we have
\begin{align*}
\langle\Phi'(u),u\rangle &= p\Phi(u)+\int_\Omega\big(pG(x,u)-g(x,u)u\big){\rm d}x \\
&\le p\Phi(u),
\end{align*}
so there exists $\eta<0$ such that for all $u\in\Phi^\eta$ we have
\begin{equation}\label{4der}
\langle\Phi'(u),u\rangle <0.
\end{equation}
By the considerations above, we see that, for all $u\in X(\Omega)\setminus\{0\}$, there exists a unique $\tau(u)\ge 1$ such that, for all $\tau\in[1,\infty)$,
\[\Phi(\tau u) \begin{cases}
>\eta & \mbox{if $1\le\tau<\tau(u)$}\\
=\eta & \mbox{if $\tau=\tau(u)$}\\
<\eta & \mbox{if $\tau>\tau(u)$}.\\
\end{cases}\]
Moreover, by the Implicit Function Theorem and \eqref{4der}, the mapping $\tau:X(\Omega)\setminus\{0\}\to[1,\infty)$ is continuous. We define a continuous deformation $h:[0,1]\times(X(\Omega)\setminus\{0\})\to X(\Omega)\setminus\{0\}$ by setting for all $(t,u)\in [0,1]\times X(\Omega)\setminus\{0\}$
\[h(t,u)=(1-t)u+t\tau(u).\]
It is immediately seen that $\Phi^\eta$ is a strong deformation retract of $X(\Omega)\setminus\{0\}$. Similarly, by radial retraction we see that $\partial B_1(0)$ is a deformation retract of $X(\Omega)\setminus\{0\}$, and $\partial B_1(0)$ is contractible (as ${\rm dim}(X(\Omega))=\infty$), so $\Phi^\eta$ is contractible.
\end{proof}

\noindent
We need to compute the critical groups of $\Phi$ at $0$. With this aim in mind, we define for all $\tau\in[0,1]$ a functional $\Phi_\tau\in C^1(X(\Omega))$ by setting for all $u\in X(\Omega)$
\[\Phi_\tau(u)=\frac{\|u\|^p}{p}-\frac{\lambda\|u\|_p^p}{p}-\int_\Omega G(x,(1-\tau)u+\tau\theta(u)){\rm d}x,\]
where $\theta\in C^1(\R,[-\delta,\delta])$ ($\delta>0$) is a non-decreasing mapping such that
\[\theta(t) = \begin{cases}
t & \mbox{if $|t|\le\delta/2$}\\
\pm\delta & \mbox{if $\pm t\ge\delta$}.
\end{cases}\]
Clearly $\Phi_0=\Phi$. Critical groups of $\Phi$ and $\Phi_1$ at $0$ coincide:

\begin{lem}\label{4hom}
$0$ is an isolated critical point of $\Phi_\tau$, uniformly with respect to $\tau\in[0,1]$, and $C^k(\Phi,0)=C^k(\Phi_1,0)$ for all $k\in\N_0$.
\end{lem}
\begin{proof}
For $\eps>0$ small enough, we have $K(\Phi)\cap\overline B_\eps(0)=\{0\}$. We prove now that, taking $\eps>0$ even smaller if necessary, we have
\begin{equation}\label{4b}
K(\Phi_\tau)\cap\overline B_\eps(0)=\{0\} \ \mbox{for all $\tau\in[0,1]$.}
\end{equation}
We argue by contradiction: assume that there exist sequences $(\tau_n)$ in $[0,1]$ and $(u_n)$ in $X(\Omega)\setminus\{0\}$ such that $\Phi'_{\tau_n}(u_n)=0$ for all $n\in\N$, and $u_n\to 0$ in $X(\Omega)$. For all $n\in\N$, we set for all $(x,t)\in\Omega\times\R$
\[g_n(x,t)=(1-\tau_n+\tau_n\theta'(t))g(x,(1-\tau_n)t+\tau_n\theta(t)),\]
where $\theta\in C^1(\R,[-\delta,\delta])$ is defined as above. By ${\bf H}_4 (i), (iii)$, for all $n\in\N$, $g_n:\Omega\times\R\to\R$ is a Carath\'eodory mapping and satisfies a.e. in $\Omega$ and for all $t\in\R$
\[|\lambda|t|^{p-2}t+g_n(x,t)|\le a'(|t|^{p-1}+|t|^{r-1}),\]
for some $a'>0$ independent of $n\in\N$. Besides, for all $n\in\N$, $u_n$ is a weak solution of the auxiliary problem
\begin{equation} \label{4c}
\begin{cases}
(-\Delta)_p^s\,u=\lambda|u|^{p-2}u+g_n(x,u) & \text{in $\Omega$} \\
u  = 0 & \text{in $\R^N \setminus \Omega$},
\end{cases}
\end{equation}
By Corollary \ref{3cor}, there exists $K>0$ (independent of $n\in\N$) such that, for all weak solution $u\in X(\Omega)$ of \eqref{4c} with $\|u\|_r<K$ we have $u\in L^\infty(\Omega)$ with $\|u\|_\infty\le K^{-1}\|u\|_r$. By the continuous embedding $X(\Omega)\hookrightarrow L^r(\Omega)$, we have $u_n\to 0$ in $L^r(\Omega)$, hance the same concergence takes place in $L^\infty(\Omega)$ as well. In particular, for $n\in\N$ big enough we have $u_n\in\overline B_\eps(0)$ and $\|u_n\|_\infty\le\delta/2$, hence by definition of $\Phi_{\tau_n}$ it is easily seen that
\[\Phi'(u_n)=\Phi_{\tau_n}'(u_n)=0,\]
i.e., $u_n\in K(\Phi)\cap\overline B_\eps(0)\setminus\{0\}$, a contradiction. So \eqref{4b} is achieved.
\vskip2pt
\noindent
For all $0\le\tau\le 1$ the functional $\Phi_\tau\in C^1(X(\Omega))$ satisfies hypotheses analogous to ${\bf H}_4$, hence by Lemma \ref{4ps} $\Phi_\tau$ satisfies ${\bf PS}$ in $\overline B_\eps(0)$. Besides, clearly the mapping $\tau\mapsto\Phi_\tau$ is continuous in $[0,1]$. So, by Proposition \ref{2hominv} we have $C^k(\Phi,0)=C^k(\Phi_1,0)$ for all $k\in\N_0$.
\end{proof}

\noindent
We prove now that $\Phi$ has a non-trivial critical group at zero for all $\lambda\in\R$, under appropriate conditions. We begin with 'small' $\lambda$'s:

\begin{lem}\label{4zerosmall}
If one of the following holds:
\begin{itemize}
\item[$(i)$] $\lambda<\lambda_1$;
\item[$(ii)$] $\lambda=\lambda_1$, and $G(x,t)\le 0$ a.e. in $\Omega$ and for all $|t|\le\delta$ (for some $\delta>0$),
\end{itemize}
then $C^k(\Phi,0)=\delta_{k,0}\,\Z_2$ for all $k\in\N_0$.
\end{lem}
\begin{proof}
By ${\bf H}_4(iii)$, for all $\eps>0$ there exists $\rho>0$ such that a.e. in $\Omega$ and for all $|t|\le\rho$
\[|g(x,t)|\le\eps|t|^{p-1}.\]
So, for all $u\in X(\Omega)$ we have by ${\bf H}_4(i)$
\begin{align*}
\Big|\int_\Omega G(x,u){\rm d}x\Big| &\le \int_{\{|u|\le\rho\}}\frac{\eps|u|^p}{p}{\rm d}x+\int_{\{|u|>\rho\}}a\Big(|u|+\frac{|u|^r}{r}\Big){\rm d}x \\
&\le \frac{\eps\|u\|_p^p}{p}+C\|u\|_r^r,
\end{align*}
which, together with the continuous embeddings $X(\Omega)\hookrightarrow L^p(\Omega),\,L^r(\Omega)$ and by arbitrarity of $\eps>0$, yields
\begin{equation}\label{4d}
\int_\Omega G(x,u){\rm d}x=o(\|u\|^p) \ \mbox{as $\|u\|\to 0$}.
\end{equation}
Now we consider separately the two cases:
\begin{itemize}
\item[$(i)$] By \eqref{4d}, we have for all $u\in X(\Omega)$
\[\Phi(u)\ge\Big(1-\frac{\lambda}{\lambda_1}\Big)\frac{\|u\|^p}{p}+o(\|u\|^p),\]
and the latter is positive for $\|u\|>0$ small enough, hence $0$ is a strict local minimizer of $\Phi$. Thus, by Lemma \ref{2critgr}, for all $k\in\N_0$ we have $C^k(\Phi,0)=\delta_{k,0}\Z_2$.
\item[$(ii)$] By Lemma \ref{4hom}, we may pass to $\Phi_1\in C^1(X(\Omega))$. For all $u\in X(\Omega)$ we have $|\theta(u(x))|\le\delta$ a.e. in $\Omega$, so
\[\Phi_1(u)\ge\Big(1-\frac{\lambda}{\lambda_1}\Big)\frac{\|u\|^p}{p}-\int_\Omega G(x,\theta(u)){\rm d}x\ge 0,\]
hence $0$ is a local minimizer of $\Phi_1$. Thus, by Lemmas \ref{2critgr} and \ref{4hom}, for all $k\in\N_0$ we have $C^k(\Phi,0)=C^k(\Phi_1,0)=\delta_{k,0}\Z_2$.
\end{itemize}
This concludes the proof.
\end{proof}

\noindent
Now we consider 'big' $\lambda$'s:

\begin{lem}\label{4zerobig}
If one of the following holds for some $k\in\N$:
\begin{itemize}
\item[$(i)$] $\lambda_k<\lambda<\lambda_{k+1}$;
\item[$(ii)$] $\lambda_k=\lambda<\lambda_{k+1}$, and $G(x,t)\ge 0$ a.e. in $\Omega$ and for all $|t|\le\delta$ (for some $\delta>0$);
\item[$(iii)$] $\lambda_k<\lambda=\lambda_{k+1}$, and $G(x,t)\le 0$ a.e. in $\Omega$ and for all $|t|\le\delta$ (for some $\delta>0$),
\end{itemize}
then $C^k(\Phi,0)\ne 0$.
\end{lem}
\begin{proof}
First we assume $(i)$. Again, \eqref{4d} holds. We prove that $\Phi$ has a cohomological local splitting near $0$ in dimension $k\in\N$ (see Definition \ref{2clsplit}). Set
\[X_+=\{u\in X(\Omega)\,:\,\|u\|^p\ge\lambda_{k+1}\|u\|_p^p\}, \quad X_-=\{u\in X(\Omega)\,:\,\|u\|^p\le\lambda_k\|u\|_p^p\}.\]
Clearly, $X_\pm$ are symmetric closed cones with $X_+\cap X_-=\{0\}$ (as $\lambda_k<\lambda_{k+1}$). Defining the manifold $\mathcal{M}$ as in \eqref{2manifold}, by Proposition \ref{2lambdak} we have
\[i\big(\mathcal{M}\cap X_-\big)=i\big(\mathcal{M}\cap (X(\Omega)\setminus X_+)\big)=k.\]
We define a mapping $h:[0,1]\times(X_-\setminus\{0\})\to(X_-\setminus\{0\})$ by setting for all $(t,u)\in[0,1]\times(X_-\setminus\{0\})$
\[h(t,u)=(1-t)u+t\frac{p^{1/p}u}{\|u\|}.\]
It is easily seen that, by means of $h$, the set $\mathcal{M}\cap X_-$ is a deformation retract of $X_-\setminus\{0\}$, so we have $i(X_-\setminus\{0\})=k$. Analogously we see that $i(X(\Omega)\setminus X_+)=k$.
\vskip2pt
\noindent
Now we prove that, for $\rho>0$ small enough,
\begin{equation}\label{4e}
\Phi(u)\le 0 \ \mbox{for all $u\in\overline B_\rho(0)\cap X_-$,} \ \Phi(u)\ge 0 \ \mbox{for all $u\in\overline B_\rho(0)\cap X_+$.}
\end{equation}
Indeed, for all $u\in X_-\setminus\{0\}$, we have by \eqref{4d}
\[\Phi(u)\le\Big(1-\frac{\lambda}{\lambda_k}\Big)\frac{\|u\|^p}{p}+o(\|u\|^p)\]
as $\|u\|\to 0$, and the latter is negative for $\|u\|>0$ small enough. Besides, for all $u\in X_+\setminus\{0\}$, we have
\[\Phi(u)\ge\Big(1-\frac{\lambda}{\lambda_{k+1}}\Big)\frac{\|u\|^p}{p}+o(\|u\|^p)\]
as $\|u\|\to 0$, and the latter is positive for $\|u\|>0$ small enough. So \eqref{4e} holds.
\vskip2pt
\noindent
Now we apply Proposition \ref{2degio} and conclude that $C^k(\Phi,0)\ne 0$.
\vskip2pt
\noindent
If we assume either $(ii)$ or $(iii)$, we can develop the same argument for $\Phi_1$ (replacing one of the strict inequalities $\lambda_k<\lambda<\lambda_{k+1}$ with the convenient sign condition on $G(x,\theta(u(x)))$ a.e. in $\Omega$). Then we apply Lemma \ref{4hom} and obtain $C^k(\Phi,0)=C^k(\Phi_1,0)\ne 0$.
\end{proof}

\noindent
Now we are ready to prove our main result:
\vskip6pt
\noindent
{\em Proof of Theorem \ref{4main}.}
We argue by contradiction, assuming
\begin{equation}\label{4f}
K(\Phi)=\{0\}.
\end{equation}
Let $\eta<0$ be as in Lemma \ref{4ps}. Since there is no critical value for $\Phi$ in $[\eta,0)$ and $\Phi$ satisfies ${\bf PS}$ in $X(\Omega)$, by the Second Deformation Theorem the set $\Phi^\eta$ is a deformation retract of $\Phi^0\setminus\{0\}$. Analogously, since there is no critical value in $(0,\infty)$, $\Phi^0$ is a deformation retract of $X(\Omega)$. So we have for all $k\in\N_0$
\[C^k(\Phi,0)=H^k(\Phi^0,\Phi^0\setminus\{0\})=H^k(X(\Omega),\Phi^\eta)=0.\]
We can easily check that, in all cases $(i)-(iii)$, one of the assumptions of either Lemma \ref{4zerosmall} or \ref{4zerobig} holds for some $k\in\N_0$, a contradiction. Thus, \eqref{4f} must be false and there exists $u\in K(\Phi)\setminus\{0\}$, which turns out to be a non-zero solution of \eqref{lprob}. \qed

\section{Multiplicity for the coercive case}
\label{coercive}

\noindent
In this section, following the methods of {\sc Liu \& Liu} \cite{LL1} (see also {\sc Liu \& Li} \cite{LL2}), we prove a multiplicity result for problem \eqref{prob}, under assumptions which make the energy functional coercive. More precisely, by a truncation argument and minimization, we prove the existence of two constant sign solutions (one positive, the other negative), then we apply Morse theory to find a third non-zero solution. In doing so, we shall need a non-local analogous of a well-known result of {\sc Garc\`{\i}a Azorero, Peral Alonso $\&$ Manfredi} \cite{GPM} about local minimizers of functionals in H\"older and Sobolev topologies, which holds under suitable regularity assumptions.
\vskip2pt
\noindent
We assume that $\Omega$ has a $C^{1,1}$ boundary. The hypotheses on the reaction term $f$ in \eqref{prob} are the following:
\begin{itemize}
\item[${\bf H}_5$] $f:\Omega\times\R\to\R$ is a Carath\'eodory mapping, $F(x,t)=\int_0^t f(x,\tau){\rm d}\tau$ for all $(x,t)\in\Omega\times\R$, and:
\begin{itemize}
\item[$(i)$] $|f(x,t)|\le a(1+|t|^{r-1})$ a.e. in $\Omega$ and for all $t\in\R$ ($a>0$, $1<r<p^*_s$);
\item[$(ii)$] $f(x,t)t\ge 0$ a.e. in $\Omega$ and for all $t\in\R$;
\item[$(iii)$] $\displaystyle\lim_{t\to 0}\frac{f(x,t)-b|t|^{q-2}t}{|t|^{p-2}t}=0$ uniformly a.e. in $\Omega$ ($b>0$);
\item[$(iv)$] $\displaystyle\limsup_{|t|\to\infty}\frac{pF(x,t)}{|t|^p}<\lambda_1$ uniformly a.e. in $\Omega$.
\end{itemize}
\end{itemize}
We define $\Phi$ as in \eqref{2phi}.
\vskip2pt
\noindent
We denote by $C^0(\overline\Omega)$ and $C^{0,\alpha}(\overline\Omega)$ ($0<\alpha<1$) the usual H\"older spaces, endowed with the norms
\[\|u\|_{C^0(\overline\Omega)}=\max_{x\in\overline\Omega}|u(x)|, \quad \|u\|_{C^{0,\alpha}(\overline\Omega)}=\|u\|_{C^0(\overline\Omega)}+\sup_{x,y\in\overline\Omega, \ x\neq y}\frac{|u(x)-u(y)|}{|x-y|^\alpha},\]
respectively. Besides, we set $d(x)={\rm dist}(x,\partial\Omega)$ for all $x\in\overline\Omega$, fix $0<\gamma<1$ and introduce the weighted H\"older spaces
\[C^0_d(\overline\Omega)=\{u\in C^0(\overline\Omega): \ ud^{-\gamma}\in C^0(\overline\Omega)\}, \quad C^{0,\alpha}_d(\overline\Omega)=\{u\in C^0(\overline\Omega): \ ud^{-\gamma}\in C^{0,\alpha}(\overline\Omega)\},\]
endowed with the norms
\[\|u\|_{C^0_d(\overline\Omega)}=\|ud^{-\gamma}\|_{C^0(\overline\Omega)}, \quad \|u\|_{C^{0,\alpha}_d(\overline\Omega)}=\|ud^{-\gamma}\|_{C^{0,\alpha}(\overline\Omega)},\]
respectively. Clearly, if $u\in C^0_d(\overline\Omega)$, then $u=0$ on $\partial\Omega$. In general, all functions that vanish at $\partial\Omega$ will be identified with their zero-extensions to $\R^N$. By the Ascoli theorem, the embedding $C^{0,\alpha}_d(\overline\Omega)\hookrightarrow C^0_d(\overline\Omega)$ is compact for all $0<\alpha<1$. $C^0_d(\overline\Omega)$ is an ordered Banach space with order cone
\[C_+=\{u\in C^0_d(\overline\Omega): \ u(x)\ge 0 \ \mbox{for all $x\in\overline\Omega$}\}.\]

\begin{lem}\label{intcone}
The interior of $C_+$, with respect to the topology of $C^0_d(\overline\Omega)$, is
\[{\rm int}(C_+)=\{u\in C^0_d(\overline\Omega): \ u(x)d(x)^{-\gamma}> 0 \ \mbox{for all $x\in\overline\Omega$}\}.\]
\end{lem}
\begin{proof}
Clearly, due to the definition of $\|\cdot\|_{C^{0,\alpha}_d(\overline\Omega)}$, the set on the right-hand side is contained in ${\rm int}(C_+)$.
\vskip2pt
\noindent
We prove that the reverse inclusion holds, arguing by contradiction. Assume that $u\in C^0_d(\overline\Omega)$ and $u(x)d(x)^{-\gamma}>0$ in $\overline\Omega$, and that there exist sequences $(u_n)$ in $C^0_d(\overline\Omega)$, $(x_n)$ in $\overline\Omega$ such that $u_n\to u$ in $C^0_d(\overline\Omega)$ and $u_n(x_n)<0$ for all $n\in\N$. Up to a relabeled subsequence, $x_n\to x$ for some $x\in\overline\Omega$, so we have
\[\frac{u_n(x_n)}{\delta(x_n)^\gamma}\to\frac{u(x)}{d(x)^\gamma}.\]
Hence, $u(x)d(x)^{-\gamma}\le 0$, a contradiction.
\end{proof}

\noindent
We will assume that the following regularity condition holds:
\begin{itemize}
\item[{\bf RC}] Let $f$ satisfy ${\bf H}_5(i), (ii)$. Then, there exist $\alpha,\gamma\in(0,1)$, only depending on the data of \eqref{prob}, such that:
\begin{itemize}
\item[$(i)$] if $u\in X(\Omega)$ is a bounded weak solution of \eqref{prob}, then $u\in C^{0,\gamma}(\overline\Omega)\cap C^{0,\alpha}_d(\overline\Omega)$ and, if $\pm u(x)>0$ in $\Omega$, then $\pm u(x)d(x)^{-\gamma}>0$ in $\overline\Omega$;
\item[$(ii)$] if $u\in X(\Omega)$ and, for all $0<\eps<1$, the restriction $\restr{\Phi}{\overline B_\eps(u)}$ attains its infimum at $u_\eps\in\overline B_\eps(u)$, then $u_\eps\in C^{0,\alpha}_d(\overline\Omega)$ and
\[\sup_{0<\eps<1}\|u_\eps\|_{C^{0,\alpha}_d(\overline\Omega)}<\infty;\]
\end{itemize}
\end{itemize}

\begin{rem}\label{boundrem}
By the results of \cite{RS, RS1}, {\bf RC} holds with $\gamma=s$ if $p=2$.
See also \cite{ISMosc} for related developments. It would be interesting to investigate the
regularity up to the boundary as
well as Hopf type lemmas for the quasi-linear case $p\neq 2$.
\end{rem}

\noindent
The main result of this section is the following:

\begin{thm}\label{5main}
If hypotheses ${\bf H}_5$ and ${\bf RC}$ hold, then problem \eqref{prob} admits at least three non-zero solutions.
\end{thm}

\noindent
We say that $u\in X(\Omega)$ is a $C^0_d(\overline\Omega)$-{\em local minimizer} of $\Phi$ if there exists $\rho>0$ such that $\Phi(u+h)\ge\Phi(u)$ for all $h\in C^0_d(\overline\Omega)$, $\|h\|_{C^0_d(\overline\Omega)}<\rho$. We say that $u$ is a $X(\Omega)$-{\em local minimizer} of $\Phi$ if there exists $\rho>0$ such that $\Phi(u+h)\ge\Phi(u)$ for all $h\in X(\Omega)$, $\|h\|<\rho$. The following result relates the two types of local minimizers:

\begin{prop}\label{xc-min}
If $u\in X(\Omega)\cap C^0_d(\overline\Omega)$ is a $C^0_d(\overline\Omega)$-local minimizer of $\Phi$, then $u$ is a $X(\Omega)$-local minimizer of $\Phi$ as well.
\end{prop}
\begin{proof}
We argue by contradiction, assuming that there exists $\rho>0$ such that $\Phi(u+h)\ge\Phi(u)$ for all $h\in C^0_d(\overline\Omega)$, $\|h\|_{C^0_d(\overline\Omega)}<\rho$, while there exists a sequence $(h_n)$ in $X(\Omega)$ such that $\|h_n\|\le 1/n$ and $\Phi(u+h_n)<\Phi(u)$, for all $n\in\N$. With no loss of generality we may assume that
\[\Phi(u+h_n)=\inf_{h\in\overline B_{1/n}(0)}\Phi(u+h),\]
so by ${\bf RC}(ii)$ we can find $\alpha,\gamma\in(0,1)$ such that the sequence $(h_n)$ is bounded in $C^{0,\alpha}_d(\overline\Omega)$. By the compact embedding $C^{0,\alpha}_d(\overline\Omega)\hookrightarrow C^0_d(\overline\Omega)$, we have, up to a relabeled subsequence, $h_n\to 0$ in $C^0_d(\overline\Omega)$ (recall that $h_n(x)\to 0$ a.e. in $\Omega$). For $n\in\N$ big enough, we have $\|h_n\|_{C^0_d(\overline\Omega)}<\rho$ and $\Phi(u+h_n)<\Phi(u)$, against our assumption.
\end{proof}

\noindent
We introduce two truncated energy functionals by setting for all $u\in X(\Omega)$
\begin{equation}\label{phipm}
\Phi_{\pm}(u)=\frac{\|u\|^p}{p}-\int_{\Omega} F(x,\pm u^\pm){\rm d}x,
\end{equation}
where $t^\pm=\max\{\pm u,0\}$. The following lemma displays some properties of $\Phi_\pm$:

\begin{lem}\label{2phipm}
We have $\Phi_\pm\in C^1(X(\Omega))$. Moreover,
\begin{itemize}
\item[$(i)$] if $u\in X(\Omega)$ is a critical point of $\Phi_\pm$, then $\pm u(x)\ge 0$ a.e. in $\Omega$;
\item[$(ii)$] $0$ is not a local minimizer of $\Phi_\pm$;
\item[$(iii)$] $\Phi_\pm$ is coercive in $X(\Omega)$.
\end{itemize}
\end{lem}
\begin{proof}
We consider $\Phi_+$, the argument for $\Phi_-$ being analogous. By ${\bf H}_5(ii)$ we have $f(x,0)=0$ a.e. in $\Omega$, so $(x,t)\to f(x,t^+)$ is Carath\'eodory and satisfies a growth condition similar to ${\bf H}_5(i)$. So, $\Phi_+\in C^1(X(\Omega))$ with derivative given for all $u,v\in X(\Omega)$ by
\[\langle\Phi_+'(u),v\rangle=\langle A(u),v\rangle-\int_\Omega f(x,u^+)v{\rm d}x.\]
Now we prove $(i)$. Assume $\Phi_+'(u)=0$ in $X(\Omega)^*$. We recall the elementary inequality
\[|\xi^--\eta^-|^p\le|\xi-\eta|^{p-2}(\xi-\eta)(\eta^--\xi^-),\]
holding for all $\xi,\eta\in\R$. Testing with $-u^-\in X(\Omega)$, we have
\[\|u^-\|^p \le \langle A(u),-u^-\rangle= -\int_\Omega f(x,u^+)u^-{\rm d}x =0.\]
Hence, $u\ge 0$ a.e. in $\Omega$.
\vskip2pt
\noindent
Now we prove $(ii)$. By ${\bf H}_5(i),(ii)$ we have a.e. in $\Omega$ and for all $t\in\R$
\begin{equation}\label{qr1}
F(x,t^+)\ge C_0|t|^q-C_1|t|^r \ (C_0,C_1>0).
\end{equation}
Consider a function $\bar u\in X(\Omega)$, $\bar u(x)>0$ a.e. in $\Omega$. For all $\tau>0$ we have
\begin{align*}
\Phi_+(\tau\bar u) &= \frac{\tau^p\|\bar u\|^p}{p}-\int_\Omega F(x,\tau\bar u){\rm d}x \\
&\le \frac{\tau^p\|\bar u\|^p}{p}-\tau^q C_0\|\bar u\|_{L^q(\Omega)}^q+\tau^r C_1\|\bar u\|_r^r,
\end{align*}
and the latter is negative for $\tau>0$ close enough to $0$. So, $0$ is not a local minimizer of $\Phi_+$.
\vskip2pt
\noindent
Finally, we prove $(iii)$. By ${\bf H}_5(iv)$, for all $\eps>0$ small enough, we have a.e. in $\Omega$ and for all $t\in\R$
\[F(x,t^+)\le\frac{\lambda_1-\eps}{p}|t|^p+C.\]
By definition of $\lambda_1$, we have for all $u\in X(\Omega)$
\begin{align*}
\Phi_+(u) &\ge \frac{\|u\|^p}{p}-\frac{\lambda_1-\eps}{p}\|u\|_{L^p(\Omega)}^p-C \\
&\ge \frac{\eps}{p\lambda_1}\|u\|^p-C,
\end{align*}
and the latter goes to $\infty$ as $\|u\|\to\infty$. So, $\Phi_+$ is coercive in $X(\Omega)$.
\end{proof}

\noindent
Now we can prove our main result:
\vskip4pt
\noindent
{\em Proof of Theorem \ref{5main}.} The functional $\Phi_+$ is coercive and sequentially weakly lower semicontinuous in $X(\Omega)$, so there exists $u_+\in X(\Omega)$ such that
\[\Phi_+(u_+)=\inf_{u\in X(\Omega)}\Phi_+(u).\]
By Lemma \ref{2phipm} $(i)$, $(ii)$ we have $u_+(x)\ge 0$ a.e. in $\Omega$ and $u_+\neq 0$. By ${\bf H}_5 (ii)$ and Proposition \ref{mp}, we have $u_+(x)>0$ a.e. in $\Omega$.
\vskip2pt
\noindent
Now we invoke ${\bf RC}(i)$ and find $\alpha,\gamma\in (0,1)$ such that $u_+\in C^{0,\alpha}_d(\overline\Omega)$ and $u(x)d(x)^{-\gamma}>0$ in $\overline\Omega$. By Lemma \ref{intcone}, then, $u_+\in {\rm int}(C_+)$. Since $\Phi$ and $\Phi_+$ agree on ${\rm int}(C_+)$, for all $u\in X(\Omega)\cap{\rm int}(K)$ we have $\Phi(u_+)\le\Phi(u)$, namely, $u_+$ is a $C^0_d(\overline\Omega)$-local minimizer of $\Phi$. By Proposition \ref{xc-min}, $u_+$ turns out to be a $X(\Omega)$-local minimizer of $\Phi$, in particular $\Phi'(u_+)=0$ in $X(\Omega)^*$.
\vskip2pt
\noindent
Similarly, we find a local minimizer $u_-\in X(\Omega)\cap(-{\rm int}(C_+))$ of $\Phi$, in particular $\Phi'(u_-)=0$ in $X(\Omega)^*$.
\vskip2pt
\noindent
From now on we argue by contradiction, assuming that
\begin{equation}\label{5abs}
K(\Phi)=\{0,u_\pm\}.
\end{equation}
Note that $\Phi(u_\pm)<\Phi(0)=0$. In particular, $0$ and $u_\pm$ are isolated critical points, so we can compute the corresponding critical groups. Clearly, since $u_\pm$ are strict local minimizers of $\Phi$, we have for all $k\in\N_0$
\begin{equation}\label{5crit1}
C^k(\Phi,u_\pm)=\delta_{k,0}\Z_2.
\end{equation}
Now we prove that for all $k\in\N_0$
\begin{equation}\label{5crit2}
C^k(\Phi,0)=0.
\end{equation}
By \eqref{qr1}, for all $u\in X(\Omega)\setminus\{0\}$ we can find $\tau(u)\in(0,1)$ such that $\Phi(\tau u)<0$ for all $0<\tau<\tau(u)$. Besides, ${\bf H}_5(iii)$ implies
\[\lim_{t\to 0}\frac{qF(x,t)-f(x,t)t}{|t|^p}=0.\]
So, for all $\eps>0$ we can find $C_\eps>0$ such that a.e. in $\Omega$ and for all $t\in\R$
\[\abs{F(x,t)-\frac{f(x,t)t}{q}}\le \eps|t|^p+C_\eps |t|^r.\]
By the relations above we have
\[\int_\Omega\Big(F(x,u)-\frac{f(x,u)u}{q}\Big){\rm d}x=o(\|u\|^p)\]
as $\|u\|\to 0$. For all $u\in X(\Omega)\setminus\{0\}$ we have
\begin{align*}
\frac{1}{q}\restr{\frac{\rm d}{{\rm d}\tau}\Phi(\tau u)}{\tau=1} &= \frac{\|u\|^p}{q}-\int_\Omega \frac{f(x,u)u}{q}{\rm d}x \\
&= \Phi(u)+\Big(\frac{1}{q}-\frac{1}{p}\Big)\|u\|^p+o(\|u\|^p)
\end{align*}
as $\|u\|\to 0$. So we can find $\rho>0$ such that, for all $u\in B_\rho(0)\setminus\{0\}$ with $\Phi(u)>0$,
\begin{equation}\label{5slope}
\restr{\frac{\rm d}{{\rm d}\tau}\Phi(\tau u)}{\tau=1}>0.
\end{equation}
This assures uniqueness of $\tau(u)$ defined as above, for all $u\in B_\rho(0)$ with $\Phi(u)>0$. We set $\tau(u)=1$ for all $u\in B_\rho(0)$ with $\Phi(u)\le 0$, so we have defined a mapping $\tau: B_\rho(0)\to (0,1]$. By \eqref{5slope} and the Implicit Function Theorem, $\tau$ turns out to be continuous. We set for all $(t,u)\in[0,1]\times B_\rho(0)$
\[h(t,u)=(1-t)u+t\tau(u)u,\]
so $h:[0,1]\times B_\rho(0)\to B_\rho(0)$ is a continuous deformation and the set $B_\rho(0)\cap\Phi^0$ is a deformation retract of $B_\rho(0)$. Similarly we deduce that $B_\rho(0)\cap\Phi^0\setminus\{0\}$ is a deformation retract of $B_\rho(0)\setminus\{0\}$. So, by recalling that ${\rm dim}(X(\Omega))=\infty$, we have
\[C^k(\Phi,0)=H^k(B_\rho(0)\cap\Phi^0,B_\rho(0)\cap\Phi^0\setminus\{0\})=H^k(B_\rho(0),B_\rho(0)\setminus\{0\})=0,\]
the last passage following from contractibility of $B_\rho(0)\setminus\{0\}$.
\vskip2pt
\noindent
Now we compute the critical groups at infinity. Reasoning as in Lemma \ref{qr1}, we see that $\Phi$ is coercive. So, being also sequentially weakly lower semicontinuous, $\Phi$ is bounded below in $X(\Omega)$. Take
\[\eta<\inf_{u\in X(\Omega)}\Phi(u),\]
then we have for all $k\in\N_0$
\begin{equation}\label{5crit3}
C^k(\Phi,\infty)=H^k(X(\Omega),\Phi^\eta)=\delta_{k,0}\Z_2.
\end{equation}
We recall Proposition \ref{morseid}. In our case, by \eqref{5crit1}, \eqref{5crit2}, and \eqref{5crit3}, the Morse identity reads as
\[\sum_{k=0}^\infty\, 2\delta_{k,0}t^k=\sum_{k=0}^\infty\,\delta_{k.0}t^k+(1+t)Q(t),\]
where $Q$ is a formal power series with coefficients in $\N_0$. Choosing $t=-1$, the relation above leads to a contradiction, hence \eqref{5abs} cannot hold. So there exists a further critical point $\tilde u\in K(\Phi)\setminus\{0,u_\pm\}$ of $\Phi$. Thus, $u_+$, $u_-$, and $\tilde u$ are pairwise distinct, non-zero weak solutions of \eqref{prob}. \qed

\begin{rem}\label{mountain}
A careful look at the proof of Theorem \ref{5main} reveals the following situation: either \eqref{prob} admits infinitely many non-zero weak solutions (if $u_\pm$ is {\em not} a strict local minimizer), or it admits at least three non-zero weak solutions, one of which, denoted $\tilde u$, is of mountain pass type, i.e. $C^1(\Phi,\tilde u)\ne 0$ (recall Proposition \ref{2critgr}). This can be seen directly, by constructing a path joining $u_+$ and $u_-$, or by contradiction. Assume that $C^1(\Phi,\tilde u)=0$. Then, from the Morse identity we would have
\[h= 1+q_0+(q_0+q_1)t+t^2 Q_1(t),\]
where $h\in\N$, $h\ge 2$, $Q(t)=q_0+q_1t+\ldots$ ($q_k\in\N$ for all $k\in\N_0$). This implies $q_0\ge 1$, hence a first-order term appears in the right-hand side, a contradiction.
\end{rem}

\noindent
Combining ingeniously the techniques seen above and in Section \ref{psuper}, we can prove a multiplicity result for problem \eqref{lprob}. Such result requires modified hypotheses (involving the {\em second} variational eigenvalue defined in \eqref{minimax}):

\begin{itemize}
\item[${\bf H'}_5$] $g:\Omega\times\R\to\R$ is a Carath\'eodory mapping, $G(x,t)=\int_0^t g(x,\tau){\rm d}\tau$, and
\begin{itemize}
\item[$(i)$] $|g(x,t)|\le a(1+|t|^{r-1})$ a.e. in $\Omega$ and for all $t\in\R$ ($a>0$, $p<r<p^*_s$);
\item[$(ii)$] $\displaystyle\lim_{t\to 0}\frac{g(x,t)}{|t|^{p-1}}=0$ uniformly a.e. in $\Omega$;
\item[$(iii)$] $\lambda_2|t|^p+g(x,t)t\ge 0$ a.e. in $\Omega$ and for all $t\in\R$;
\item[$(iv)$] $\displaystyle\lim_{|t|\to\infty}\frac{\lambda|t|^p+pG(x,t)}{|t|^p}<\lambda_1$ uniformly a.e. in $\Omega$.
\end{itemize}
\end{itemize}
Note that, by ${\bf H'}_5 (iv)$, we have in particular
\[\lim_{|t|\to\infty}\frac{G(x,t)}{|t|^p}=-\infty,\]
and that we places ourselves again in the coercive case. Our multiplicity result is the following:

\begin{thm}\label{5per}
If ${\bf H'}_5$, {\bf RC}, and one of the following hold:
\begin{itemize}
\item[$(i)$] $\lambda>\lambda_2$, $\lambda\notin(\lambda_k)$;
\item[$(ii)$] $\lambda\ge\lambda_2$ and $G(x,t) \ge 0$ for a.e. in $\Omega$ and for all $|t| \le \delta$ (for some $\delta > 0$);
\item[$(iii)$] $\lambda\ge\lambda_3$ and $G(x,t) \le 0$ for a.e. in $\Omega$ and for all $|t| \le \delta$ (for some $\delta > 0$),
\end{itemize}
then problem \eqref{lprob} admits at least three non-zero solutions.
\end{thm}
\begin{proof}
Clearly $0\in K(\Phi)$. Reasoning as in the proof of Theorem \ref{5main}, we find $u_\pm\in K(\Phi)\setminus\{0\}$ with $C^k(\Phi,u_\pm) =\delta_{k,0}\Z_2$ and see that $C^k(\Phi,\infty) = \delta_{k,0}\Z_2$ for all $k\in\N_0$. Besides, in all cases $(i) - (iii)$, we argue as in Lemma \ref{4zerobig} and find $k\ge 2$ such that $C^k(\Phi,0) \ne 0$. Then we apply \cite[Proposition 3.28(ii)]{PAO} and deduce that there exists $\tilde u\in K(\Phi)$ such that either $\Phi(\tilde u)<0$ and $C^{k-1}(\Phi,\tilde u)\ne 0$, or $\Phi(\tilde u)>0$ and $C^{k+1}(\Phi,\tilde u)\ne 0$. Clearly $\tilde u\ne 0$. Moreover, since $k\ge 2$, it follows at once that $\tilde u\ne u_\pm$. Thus, $u_+$, $u_-$, and $\tilde u$ are pairwise distinct, non-zero weak solutions of \eqref{lprob}.
\end{proof}

\section{Asymptotically $p$-linear case}
\label{asymptotically}

\noindent
In this section we deal with problem \eqref{prob}, in the case when $f(x,\cdot)$ is {\em asymptotically $p$-linear} at infinity, i.e.
\[\lim_{|t|\to\infty}\frac{f(x,t)}{|t|^{p-2}t}=\lambda\]
uniformly a.e. in $\Omega$, for some $\lambda\in(0,\infty)$. The problem is said to be of {\em resonant} type if $\lambda\in\sigma(s,p)$, of {\em non-resonant} type otherwise. The two cases require different techniques to prove the existence of a non-zero solution (analogous results in the non-resonant case for the $p$-Laplacian were proved by {\sc Liu \& Li} \cite{LL2}, on the basis of {\sc Perera} \cite{P}). If $f(x,\cdot)$ has a $p$-linear behavior at zero as well, but with a different slope, then we can prove the existence of two non-zero solutions, one non-negative, the other non-positive, both in the resonant and non-resonant case, by employing a truncation method (see {\sc Zhang, Li, Liu \& Feng} \cite{ZLLF} and {\sc Li \& Zhou} \cite{LZ} for the $p$-Laplacian case).
\vskip2pt
\noindent
We state here our first set of hypotheses:
\begin{itemize}
\item[${\bf H}_6$] $f:\Omega\times\R\to\R$ is a Carath\'eodory mapping, $F(x,t)=\int_0^t f(x,\tau){\rm d}\tau$ for all $(x,t)\in\Omega\times\R$, and:
\begin{itemize}
\item[$(i)$] $|f(x,t)|\le a(1+|t|^{r-1})$ a.e. in $\Omega$ and for all $t\in\R$ ($a>0$, $1<r<p^*_s$);
\item[$(ii)$] $\displaystyle\lim_{|t|\to\infty}\frac{f(x,t)}{|t|^{p-2}t}=\lambda$ uniformly a.e. in $\Omega$ ($\lambda>0$);
\item[$(iii)$] $\displaystyle\lim_{t\to 0}\frac{f(x,t)-b|t|^{q-2}t}{|t|^{p-2}t}=0$ uniformly a.e. in $\Omega$ ($b>0$, $1<q<p$).
\end{itemize}
\end{itemize}
Clearly, ${\bf H}_6 (iii)$ implies that $f(x,0)=0$ a.e. in $\Omega$, so \eqref{prob} admits the zero solution. We seek non-zero solutions, so with no loss of generality we may assume that {\em all} critical points of the energy functional $\Phi\in C^1(X(\Omega))$ (defined as in \eqref{2phi}) are isolated.
\vskip2pt
\noindent
First we introduce our existence result for the non-resonant case:

\begin{thm}\label{6ex1}
If ${\bf H}_6$ holds with $\lambda\notin\sigma(s,p)$, then problem \eqref{prob} admits at least a non-zero solution.
\end{thm}
\begin{proof}
We first consider the case $0<\lambda<\lambda_1$. In such case, $\Phi$ is coercive and sequentially weakly lower semi-continuous, so it has a global minimizer $u\in K(\Phi)$. By Proposition \ref{2critgr} $(i)$ we have $C^k(\Phi,u)=\delta_{k,0}\Z_2$ for all $k\in\N_0$. If $\lambda>\lambda_1$, then we can find $k\in\N$ such that $\lambda_k<\lambda<\lambda_{k+1}$. By \cite[Theorem 5.7]{PAO}, there exists $u\in K(\Phi)$ such that $C^k(\Phi,u)\ne 0$. In either case, we have found $u\in K(\Phi)$ with a non-trivial critical group.
\vskip2pt
\noindent
By ${\bf H}_6 (iii)$, reasoning as in the proof of Theorem \ref{5main}, we can see that $C^k(\Phi,0)=0$ for all $k\in\N_0$, so $u\ne 0$.
\end{proof}

\noindent
In the study of the resonant case, we meet a significant difficulty: the energy functional $\Phi$ need not satisfy ${\bf PS}$. So, we need to introduce additional conditions in order to ensure compactness of critical sequences. We set for all $(x,t)\in\Omega\times\R$
\[H(x,t)=pF(x,t)-f(x,t)t.\]
We have the following existence result:

\begin{thm}\label{6ex2}
If ${\bf H}_6$ holds with $\lambda\in\sigma(s,p)$, and there exist $k\in\N$, $h_0\in L^1(\Omega)$ such that one of the following holds:
\begin{itemize}
\item[$(i)$] $\lambda_k<\lambda\le\lambda_{k+1}$, $H(x,t)\le -h_0(x)$ a.e. in $\Omega$ and for all $t\in\R$, and
\[\lim_{|t|\to\infty}H(x,t)=-\infty\]
uniformly a.e. in $\Omega$;
\item[$(iii)$] $\lambda_k\le\lambda<\lambda_{k+1}$, $H(x,t)\ge h_0(x)$ a.e. in $\Omega$ and for all $t\in\R$, and
\[\lim_{|t|\to\infty}H(x,t)=\infty\]
uniformly a.e. in $\Omega$,
\end{itemize}
then problem \eqref{prob} admits at least a non-zero solution.
\end{thm}
\begin{proof}
Since $\lambda\in\sigma(s,p)$, by Proposition \ref{spec} $(i)$ there exists $k\in\N$ such that $\lambda\in[\lambda_k,\lambda_{k+1}]$, and the latter is a non-degenerate interval. We assume $(i)$. We aim at applying \cite[Theorem 5.9]{PAO}, but first we need to verify some technical conditions. Set for all $u\in X(\Omega)$
\[\Psi(u)=\Phi(u)-\frac{1}{p}\langle\Phi'(u),u\rangle=-\frac{1}{p}\int_\Omega H(x,u)\,{\rm d}x.\]
Then, for all $u\in X(\Omega)$ we have
\[\Psi(u)\ge\frac{1}{p}\|h\|_1,\]
hence $\Psi$ is bounded below in $X(\Omega)$. Moreover, if $(u_n)$ is a sequence in $X(\Omega)$ such that $\|u_n\|\to \infty$, $v_n=\|u_n\|^{-1}u_n\to v$ in $X(\Omega)$, then in particular we have $v_n(x)\to v(x)$ a.e. in $\Omega$. So, by the Fatou Lemma we have for all $n\in\N$, $\tau\ge 1$
\[\Psi(\tau u_n) = -\frac{1}{p}\int_\Omega H(x,\|u_n\|\tau v_n)\,{\rm d}x,\]
and the latter tends to $\infty$ as $n\to\infty$. We conclude that condition $(H_+)$ holds (see \cite{PAO}, p. 82). So, by \cite[Theorem 5.9]{PAO}, $\Phi$ satisfies ${\bf C}$ and there exists $u\in K(\Phi)$ such that $C^k(\Phi,u)\ne 0$. Reasoning as in the proof of Theorem \ref{6ex1} we see that $u\ne 0$. Thus, \eqref{prob} has a non-zero solution.
\vskip2pt
\noindent
The argument for the case $(ii)$ is analogous.
\end{proof}

\begin{rem}\label{zerosol}
We note that, if we only assume ${\bf H}_6 (i),(ii)$, by the same arguments used in Theorems \ref{6ex1} and \ref{6ex2} we can prove the existence of a (possibly zero) solution. This is still a valuable information, since we have no condition on $f(\cdot,0)$.
\end{rem}

\noindent
In the remaining part of the section we deal with the case of a reaction term $f$ which behaves $p$-linearly both at infinity and at zero, but with different slopes. Our hypotheses are the following:
\begin{itemize}
\item[${\bf H'}_6$] $f:\Omega\times\R\to\R$ is a Carath\'eodory mapping, $F(x,t)=\int_0^t f(x,\tau){\rm d}\tau$ for all $(x,t)\in\Omega\times\R$, and:
\begin{itemize}
\item[$(i)$] $|f(x,t)|\le a(1+|t|^{r-1})$ a.e. in $\Omega$ and for all $t\in\R$ ($a>0$, $1<r<p^*_s$);
\item[$(ii)$] $\displaystyle\lim_{|t|\to\infty}\frac{f(x,t)}{|t|^{p-2}t}=\lambda$ uniformly a.e. in $\Omega$ ($\lambda>\lambda_1$);
\item[$(iii)$] $\displaystyle\lim_{t\to 0}\frac{f(x,t)}{|t|^{p-2}t}=\mu$ uniformly a.e. in $\Omega$ ($0<\mu<\lambda_1$).
\end{itemize}
\end{itemize}
By ${\bf H'}_6 (iii)$, we have $f(x,0)=0$ a.e. in $\Omega$, hence \eqref{prob} admits the zero solution. For non-zero solutions, we have the following multiplicity result:

\begin{thm}\label{6mult}
If ${\bf H'}_6$ holds, then problem \eqref{prob} admits at least two non-zero solutions, one non-negative, the other non-positive.
\end{thm}

\begin{rem}\label{strict}
If, beside ${\bf H'}_6$, we also assume a sign condition of the type $f(x,t)t\ge 0$ a.e. in $\Omega$ and for all $t\in\R$, then by applying Proposition \ref{mp} we can prove the existence of a strictly positive and of a strictly negative solution.
\end{rem}

\noindent
Since $f(\cdot,0)=0$, we can define truncated energy functionals $\Phi_\pm\in C^1(X(\Omega))$ as in \eqref{phipm}. We have for all $u,v\in X(\Omega)$
\[\langle\Phi'_\pm(u),v\rangle=\langle A(u)\mp\lambda B_\pm(u),v\rangle-\int_\Omega g_\pm(x,u)v\,{\rm d}x,\]
where we set for all $(x,t)\in\Omega\times\R$
\[g_\pm(x,t)=f(x,\pm t^\pm)\mp\lambda(t^\pm)^{p-1}\]
and for all $u,v\in X(\Omega)$
\[\langle B_\pm(u),v\rangle=\int_\Omega (u^\pm)^{p-1}v\, {\rm d}x.\]
By the compact embedding $X(\Omega)\hookrightarrow L^p(\Omega)$, $B_\pm:X(\Omega)\to X(\Omega)^*$ is a completely continuous operator.

\begin{lem}\label{6bpm}
There exists $\rho>0$ such that $\|A(u)\mp\lambda B_\pm(u)\|_*\ge\rho\|u\|^{p-1}$ for all $u\in X(\Omega)$.
\end{lem}
\begin{proof}
We deal with $A-\lambda B_+$ (the argument for $A+\lambda B_-$ is analogous). We argue by contradiction: let $(u_n)$, $(\eps_n)$ be sequences in $X(\Omega)$ and in $(0,\infty)$, respectively, such that $\eps_n\to 0$ as $n\to \infty$, and for all $n\in\N$
\[\|A(u_n)-\lambda B_+(u_n)\|_*=\eps_n\|u_n\|^{p-1}.\]
Since $A-\lambda B_+$ is $(p-1)$-homogeneous, we may assume $\|u_n\|=1$ for all $n\in\N$. So $(u_n)$ is bounded, and passing to a relabeled subsequence we have $u_n\rightharpoonup u$ in $X(\Omega)$, $u_n\to u$ in $L^p(\Omega)$ and $(u_n^+)^{p-1}\to (u^+)^{p-1}$ in $L^{p'}(\Omega)$. For all $n\in\N$ we have
\begin{align*}
\big|\langle A(u_n),u_n-u\rangle\big| &\le \big|\langle A(u_n)-\lambda B_+(u_n),u_n-u\rangle\big|+\lambda\big|\langle B_+(u_n),u_n-u\rangle\big| \\
&\le \eps_n\|u_n-u\|+\lambda\|u_n^+\|_p^{p-1}\|u_n-u\|_p,
\end{align*}
and the latter tends to $0$ as $n\to\infty$. By the $({\bf S})$-property of the operator $A$, we deduce $u_n\to u$ in $X(\Omega)$. So, $\|u\|=1$ and for all $v\in X(\Omega)$
\[\langle A(u),v\rangle=\lambda\int_\Omega (u^+)^{p-1}v\, {\rm d}x.\]
Reasoning as in Lemma \ref{2phipm} $(i)$, we see that $u(x)\ge 0$ a.e. in $\Omega$. By Proposition \ref{mp}, then, we have $u(x)>0$ a.e. in $\Omega$. Thus, $u$ turns out to be a positive $\lambda$-eigenfunction with $\lambda>\lambda_1$, against Proposition \ref{spec} $(iii)$. This concludes the proof.
\end{proof}

\noindent
We point out the following technical lemma:

\begin{lem}\label{6ps}
$\Phi_\pm\in C^1(X(\Omega))$ satisfies ${\bf PS}$ in $X(\Omega)$.
\end{lem}
\begin{proof}
We deal with $\Phi_+$ (the argument for $\Phi_-$ is analogous). Let $(u_n)$ be a sequence in $X(\Omega)$ such that $(\Phi_+(u_n))$ is bounded in $\R$ and $\Phi'_+(u_n)\to 0$ in $X(\Omega)^*$. We prove that $(u_n)$ is bounded, arguing by contradiction: assume that (passing if necessary to a subsequence) $\|u_n\|\to\infty$ as $n\to\infty$. Let $\rho>0$ be as in Lemma \ref{6bpm}. By ${\bf H'}_6 (i),(iii)$ we have $g_+(x,t)=o(t^{p-1})$ as $t\to\infty$, so there exists $C_\rho>0$ such that a.e. in $\Omega$ and for all $t\in\R$
\[|g_+(x,t)|\le\frac{\rho\lambda_1}{2}(t^+)^{p-1}+C_\rho.\]
For all $n\in\N$, $v\in X(\Omega)$ we have
\begin{align*}
\big|\langle A(u_n)-\lambda B_+(u_n),v\rangle\big| &\le \big|\langle \Phi'_+(u_n),v\rangle\big|+\int_\Omega |g_+(x,u) v|\,{\rm d}x \\
&\le \|\Phi'_+(u_n)\|_*\|v\|+\frac{\rho\lambda_1}{2}\|u_n^+\|_p^{p-1}\|v\|_p+C_\rho\|v\|_1 \\
&\le \|\Phi'_+(u_n)\|_*\|v\|+\frac{\rho}{2}\|u_n\|^{p-1}\|v\|+C\|v\|.
\end{align*}
So, using also Lemma \ref{6bpm}, we have for all $n\in\N$
\[\rho\|u_n\|^{p-1}\le\|A(u_n)-\lambda B_+(u_n)\|_*\le\frac{\rho}{2}\|u_n\|^{p-1}+o(\|u_n\|^{p-1}),\]
a contradiction as $n\to\infty$. Thus, $(u_n)$ is bounded, and as in the proof of Proposition \ref{2ps} we conclude that $(u_n)$ has a convergent subsequence.
\end{proof}

\noindent
Now we are ready to prove our main result:
\vskip6pt
\noindent
{\em Proof of Theorem \ref{6mult}.} In ${\bf H'}_6 (i)$ we can always set $p<r<p^*_s$. Choose real numbers $\mu<\alpha<\lambda_1<\beta<\lambda$. By ${\bf H'}_6 (i),(iii)$ there exists $C_\alpha>0$ such that a.e. in $\Omega$ and for all $t\in\Omega$
\[|F(x,t^+)|\le\frac{\alpha}{p}|t|^p+C_\alpha |t|^r.\]
For all $u\in X(\Omega)$ we have
\begin{align*}
\Phi_+(u) &\ge \frac{\|u\|^p}{p}-\frac{\alpha}{p}\|u\|_p^p-C_\alpha\|u\|_r^r \\
&\ge \Big(1-\frac{\alpha}{\lambda_1}\Big)\frac{\|u\|^p}{p}-C\|u\|^r.
\end{align*}
So, we can find $R,c>0$ such that
\begin{equation}\label{6mp}
\inf_{u\in\partial B_R(0)}\Phi_+(u)=c.
\end{equation}
By ${\bf H'}_6 (i),(ii)$ there exists $C_\beta>0$ such that a.e. in $\Omega$ and for all $t\in\Omega$
\[F(x,t^+)\ge\frac{\beta}{p}(t^+)^p-C_\beta.\]
Let $u_1\in X(\Omega)$ be a positive $\lambda_1$-eigenfunction (recall Proposition \ref{spec} $(ii)$), then for all $\tau>0$ we have
\begin{align*}
\Phi_+(\tau u_1) &= \frac{\tau^p\|u_1\|^p}{p}-\int_\Omega F(x,\tau u_1) {\rm d}x \\
&\le \frac{\tau^p\|u_1\|^p}{p}-\frac{\beta\tau^p}{p}\|u_1\|_p^p+C \\
&\le \tau^p\Big(1-\frac{\beta}{\lambda_1}\Big)\frac{\|u_1\|^p}{p}+C,
\end{align*}
and the latter tends to $-\infty$ as $\tau\to\infty$. So, $\Phi_+$ exhibits the 'mountain pass geometry'. By Lemma \ref{6ps}, $\Phi_+$ satisfies ${\bf PS}$ in $X(\Omega)$. Hence, by the Mountain Pass Theorem, there exists $u_+\in K(\Phi_+)$ such that $\Phi_+(u_+)\ge c$, with $c$ as in \eqref{6ps}. In particular, then, $u_+\ne 0$. Reasoning as in the proof of Lemma \ref{2phipm} $(i)$ we see that $u_+(x)\ge 0$ a.e. in $\Omega$, hence $u\in K(\Phi)$ turns out to be a non-negative, non-zero solution of \eqref{prob}.
\vskip2pt
\noindent
In a similar way, working on $\Phi_-$, we produce a non-positive, non-zero solution $u_-$ of \eqref{prob} (in particular, $u_+\ne u_-$). \qed

\begin{rem}\label{6not}
We could have denoted $f(x,t)=\lambda|t|^{p-2}t+g(x,t)$ for all $(x,t)\in\Omega\times\R$ as in Section \ref{psuper}, with $g(x,t)=o(|t|^{p-1})$ at infinity. But in Theorem \ref{6mult}, this would have lead to unnatural condition on the behavior of $g(x,\cdot)$ at zero.
\end{rem}

\section{Poho\v zaev identity and nonexistence}
\label{pohoz}

\noindent
In this section we discuss possible non-existence results for problems involving the 
operator $(-\Delta)^s_p$ via a convenient Poho\v zaev identity. We focus first on the autonomous equation
\begin{equation} \label{7prob}
(- \Delta)_p^s\, u  = f(u) \quad \text{in $\R^n$,}
\end{equation}
where $0<s<1<p<N$, $f\in C(\R)$, and we set for all $(x,t)\in\Omega\times\R$
\[F(t)=\int_0^t f(\tau){\rm d}\tau.\]
A {\em weak solution} of \eqref{7prob} is a function $u\in W^{s,p}(\R^N)$ such that for all $v\in W^{s,p}(\R^N)$
\[\langle A(u),v\rangle=\int_{\R^N} f(u)v\,{\rm d}x.\]
As usual, weak solutions of \eqref{7prob} can be detected as the critical points of an energy functional $\Phi\in C^1(W^{s,p}(\R^N))$ defined by setting for all $u\in W^{s,p}(\R^N)$
\[\Phi(u)=\frac{[u]_{s,p}^p}{p}-\int_{\R^N}F(u){\rm d}x,\]
by assuming convenient growth conditions on $f$.
\vskip2pt
\noindent
Let $u\in W^{s,p}(\R^N)$ be a weak solution of \eqref{7prob}. We define a continuous path $\gamma_u:]0,1]\to W^{s,p}(\R^N)$ by setting 
for all $\theta\in]0,1]$ and $x\in\R^N$
\[
\gamma_u(\theta)(x):=u(\theta x).
\]
A simple scaling argument shows that, for all $\theta\in]0,1]$,
\[
\Phi\circ\gamma_u(\theta)=\frac{\theta^{sp-N}}{p}[u]_{s,p}^p-\theta^{-N}\int_{\R^N} F(u){\rm d}x,
\]
and
\[\frac{\rm d}{{\rm d}\theta}\restr{\Phi\circ\gamma_u(\theta)}{\theta=1}=\frac{sp-N}{p}[u]_{s,p}^p+N\int_{\R^N} F(u){\rm d}x.\]
The general form of the {\em Poho\v{z}aev identity} for problem \eqref{7prob} is
\[\frac{\rm d}{{\rm d}\theta}\restr{\Phi\circ\gamma_u(\theta)}{\theta=1}=0,\]
which, in our case, is easily seen to be equivalent to the following formula:
\begin{equation}\label{poho}
\int_{\R^N}\Big(\frac{N-sp}{Np}f(u)u-F(u)\Big){\rm d}x=0.
\end{equation}
Identity \eqref{poho} is a major tool to prove non-existence results for problem \eqref{7prob}. Nevertheless, it requires a more sophisticated machinery, as we need to deduce that
\[\langle\Phi'(u),(x\cdot\nabla u)\rangle=0,\]
and hence we need good regularity results in order to justify that $v=x\cdot\nabla u$ 
is an admissible test function for problem \eqref{7prob}. Such regularity theory is not available yet.

\begin{rem}\label{changwang}
In the semi-linear case $p=2$, for which the regularity theory is well established, 
a version of \eqref{poho} has recently be proved by {\sc Chang \& Wang} \cite[Proposition 4.1]{CW}. Namely, for any weak solution $u\in H^s(\R^N)$ of
\[(- \Delta)^s\, u  = f(u) \quad \text{in $\R^n$,}\]
we have
\[\int_{\R^N}\Big(\frac{N-2s}{Np}f(u)u-F(u)\Big){\rm d}x=0.\]
\end{rem}

\noindent
Now we introduce a bounded, smooth domain $\Omega$ and couple \eqref{7prob} with zero Dirichlet conditions outside $\Omega$, i.e., we consider the problem
\begin{equation} \label{7dir}
\begin{cases}
(- \Delta)_p^s\, u  = f(u) & \text{in $\Omega$} \\
u  = 0 & \text{in $\R^N \setminus \Omega$.}
\end{cases}
\end{equation}
Obviously, a {\em weak solution} of \eqref{7dir} is understood as $u\in X(\Omega)$ such that, for all $v\in X(\Omega)$,
\[
\langle A(u),v\rangle=\int_\Omega f(u)v\,{\rm d}x.
\]
In this framework, things become even more involved due to the presence of a boundary contribution in the identity. A reasonable candidate to play the role of \eqref{poho}, for a weak solution $u\in X(\Omega)$ of \eqref{7dir}, is the following formula:
\begin{equation}\label{pohodir}
\int_\Omega\Big(\frac{N-sp}{Np}f(u)u-F(u)\Big){\rm d}x=-M\int_{\partial\Omega}\Big(\frac{u}{d(x)^\gamma}\Big)^2(x\cdot\nu)\,{\rm d}\sigma,
\end{equation}
where $M>0$, $\gamma\in(0,1)$ depend on $s$, $p$, and $N$, $\nu$ denotes the outward normal unit vector to $\partial\Omega$ (see ${\bf RC}$ and the related discussion in Section \ref{coercive}). If $\Omega$ is star-shaped, by means of \eqref{pohodir} one should be able to prove some non-existence results for problem \eqref{7dir} of the following type:

\begin{prop}\label{7nonex}
If $f\in C(\R)$ satisfies for all $t\in\R$
\[\frac{N-sp}{Np}f(t)t-F(t)\ge 0,\]
then problem \eqref{7dir} does not admit any positive bounded weak solution. Moreover, if the inequality above is strict for all $t\in\R\setminus\{0\}$, then problem \eqref{7dir} does not admit any non-zero bounded solution.
\end{prop}

\noindent
If we reduce ourselves to the pure power-type reaction terms $f(t)=|t|^{r-2}t$ ($r>0$), then the assumption of Proposition \ref{7nonex} becomes $r\ge p^*_s$, so non-zero solutions are ruled out for $r>p^*_s$ (as expected).

\begin{rem}\label{rosserra1}
In the semi-linear case $p=2$, {\sc Ros Oton \& Serra} \cite[Theorem 1.1]{RS1} have proved the following special case of \eqref{pohodir}: if $u\in H^s(\R^N)$ is a weak solution of
\[\begin{cases}
(- \Delta)^s\, u  = f(u) & \text{in $\Omega$} \\
u  = 0 & \text{in $\R^N \setminus \Omega$,}
\end{cases}\]
then
\[\int_\Omega\Big(\frac{N-2s}{2N}f(u)u-F(u)\Big){\rm d}x=-\frac{\Gamma(1+s)^2}{2N}\int_{\partial\Omega}\Big(\frac{u}{d(x)^s}\Big)^2(x\cdot\nu)\,{\rm d}\sigma.\]
Such identity has been applied to prove non-existence results of the type discussed above (see \cite[Corollaries 1.2, 1.3]{RS1}).
\end{rem}

\begin{rem}\label{rosserra2}
In the non-linear case $p\ne 2$, other approaches may lead to non-existence results. For instance, again {\sc Ros Oton \& Serra} \cite{RS2} have obtained the following result for problem \eqref{prob}: if $f\in C^{0,1}_{\rm loc}(\overline\Omega\times\R)$ is of supercritical type, i.e., if
\[(N-sp)f(x,t)t-NpF(x,t)-p x\cdot F_x(x,t)>0\]
holds for all $(x,t)\in\overline\Omega\times\R$, then \eqref{prob} does not admit any non-zero bounded solution 
which belong $C^{1,\alpha}(\Omega)$ ($0<\alpha<1$).
\end{rem}

\bigskip

\end{document}